\documentclass[a4paper,11pt, british, reqno]{amsart}

\usepackage{babel}
\usepackage{amssymb}
\usepackage{enumitem}
\usepackage{hyperref}
\usepackage[utf8]{inputenc}
\usepackage{newunicodechar}
\usepackage{mathtools}
\usepackage{varioref}
\usepackage[centering]{geometry}
\usepackage[arrow,curve,matrix]{xy}

\usepackage{colortbl}
\usepackage{graphicx}
\usepackage{tikz}
\usepackage{tikz-cd}

\usepackage{mathrsfs}

\usepackage{longtable} 
\usepackage{caption}
\newcolumntype{M}[1]{>{\centering\arraybackslash}m{#1}} 

\definecolor{linkred}{rgb}{0.7,0.2,0.2}
\definecolor{linkblue}{rgb}{0,0.2,0.6}

\numberwithin{figure}{section}

\setdescription{labelindent=\parindent, leftmargin=2\parindent}
\setitemize[1]{labelindent=\parindent, leftmargin=2\parindent}
\setenumerate[1]{labelindent=0cm, leftmargin=*, widest=iiii}

\DeclareFontFamily{OMS}{rsfs}{\skewchar\font'60}
\DeclareFontShape{OMS}{rsfs}{m}{n}{<-5>rsfs5 <5-7>rsfs7 <7->rsfs10 }{}
\DeclareSymbolFont{rsfs}{OMS}{rsfs}{m}{n}
\DeclareSymbolFontAlphabet{\scr}{rsfs}
\DeclareSymbolFontAlphabet{\scr}{rsfs}

\DeclareMathOperator{\Id}{Id}

\DeclareMathOperator{\Pic}{Pic}
\DeclareMathOperator{\rank}{rank}

\DeclareMathOperator{\red}{red}

\DeclareMathOperator{\sing}{sing}

\DeclareMathOperator{\Sym}{Sym}
\DeclareMathOperator{\supp}{supp}

\DeclareMathOperator{\mult}{mult}

\DeclareMathOperator{\Chow}{Chow}
\DeclareMathOperator{\Exc}{Exc}
\DeclareMathOperator{\Nef}{Nef}
\DeclareMathOperator{\Pseff}{Pseff}

\newcommand{\sE}{\scr{E}}
\newcommand{\sF}{\scr{F}}
\newcommand{\sG}{\scr{G}}
\newcommand{\sH}{\scr{H}}
\newcommand{\sHom}{\scr{H}\negthinspace om}
\newcommand{\sI}{\scr{I}}

\newcommand{\sL}{\scr{L}}
\newcommand{\sM}{\scr{M}}

\newcommand{\sO}{\scr{O}}

\newcommand{\sQ}{\scr{Q}}

\newcommand{\cK}{\mathcal K}

\newcommand{\bC}{\mathbb{C}}

\newcommand{\bN}{\mathbb{N}}

\newcommand{\bP}{\mathbb{P}}
\newcommand{\bQ}{\mathbb{Q}}
\newcommand{\bR}{\mathbb{R}}

\newcommand{\bZ}{\mathbb{Z}}

\theoremstyle{plain}
\newtheorem{thm}{Theorem}[section]

\newtheorem{cor}[thm]{Corollary}
\newtheorem{defn}[thm]{Definition}

\newtheorem{lem}[thm]{Lemma}

\newtheorem{prop}[thm]{Proposition}

\theoremstyle{remark}

\newtheorem{claim}[thm]{Claim}
\newtheorem{c-n-d}[thm]{Claim and Definition}

\newtheorem{example}[thm]{Example}

\newtheorem{rem}[thm]{Remark}
\newtheorem{question}[thm]{Question}
\newtheorem*{rem-nonumber}{Remark}

\numberwithin{equation}{thm}

\setlist[enumerate]{label=(\thethm.\arabic*), before={\setcounter{enumi}{\value{equation}}}, after={\setcounter{equation}{\value{enumi}}}}

\newcommand{\factor}[2]{\left. \raise 2pt\hbox{$#1$} \right/\hskip -2pt\raise -2pt\hbox{$#2$}}
\author{Jie Liu} %
\address{Jie Liu, Institute of Mathematics, Academy of Mathematics and Systems Science, Chinese Academy of Sciences, Beijing, 100190, China}
\email{\href{jliu@amss.ac.cn}{jliu@amss.ac.cn}}
\urladdr{\href{http://www.jliumath.com}{http://www.jliumath.com}}


\makeatletter
\@namedef{subjclassname@2020}{2020 Mathematics Subject Classification}
\makeatother

\title{Fano foliations with small algebraic ranks}
\date{\today}

\makeatletter

\hypersetup{
	pdfauthor={\authors},
	pdftitle={\@title},
	pdfsubject={\@subjclass},
	pdfkeywords={\@keywords},
	pdfstartview={Fit},
	pdfpagemode={UseOutlines},
	pdfpagelayout={OneColumn},
	bookmarks,
	colorlinks,
	linkcolor=linkblue,
	citecolor=linkred,
	urlcolor=linkred
}
\makeatother

\begin{document}
	
\begin{abstract}
	In this paper we study the algebraic ranks of foliations on $\bQ$-factorial normal projective varieties. We start by establishing a Kobayashi-Ochiai's theorem for Fano foliations in terms of algebraic rank. We then investigate the local positivity of the anti-canonical divisors of foliations, obtaining a lower bound for the algebraic rank of a foliation in terms of Seshadri constant. We describe those foliations whose algebraic rank slightly exceeds this bound and classify Fano foliations on smooth projective varieties attaining this bound. Finally we construct several examples to illustrate the general situation, which in particular allow us to answer a question asked by Araujo and Druel on the generalised indices of foliations.
\end{abstract}

\maketitle
\tableofcontents

\section{Introduction}

A \emph{Fano variety} is a normal projective variety $X$ such that $-K_X$ is an ample $\bQ$-Cartier divisor. The \emph{Fano index} of $X$ is the largest positive rational number $\iota_X$ such that $-K_X\sim_{\bQ} \iota_X H$ for a Cartier divisor $H$ on $X$ and it can be viewed as an invariant measuring the global positivity of the anti-canonical divisor $-K_X$ of $X$. When $X$ is an $n$-dimensional smooth Fano variety, a classical theorem of Kobayashi and Ochiai \cite{KobayashiOchiai1973} asserts that $\iota_X\leq n+1$ with equality if and only if $X\cong \bP^n$ and this result can also be generalised to singular Fano varieties, see \cite{Fujita1989,Maeda1990,AraujoDruel2014} and Theorem \ref{t.Kobayashi-Ochiai}. 

Similar ideas can be also applied to the context of \emph{foliations} on normal projective varieties (with mild singularities). One of the central problem in the theory of foliations is to find conditions that guarantee the existence of algebraic leaves and the notion of \emph{algebraic rank} was introduced  to measure the algebraicity of leaves \cite[Definition 2.4]{AraujoDruel2019}. More precisely, the algebraic rank $r^a$ of a foliation $\sF\subsetneq T_X$ on a normal projective variety $X$ is the maximum dimension of an algebraic subvariety through a general point of $X$ that is tangent to $\sF$. Moreover, these maximal algebraic subvarieties tangent to $\sF$ are actually the leaves of a subfoliation $\sF^a$ of $\sF$ which is called the \emph{algebraic part} of $\sF$. We refer the reader to \cite[Lemma 2.4]{LorayPereiraTouzet2018} and Definition \ref{d.algebraic-part} for more details.

Let $\sF\subsetneq T_X$ be a foliation on a normal projective variety $X$. The \emph{canonical class} of $\sF$ is a Weil divisor $K_{\sF}$ such that $\sO_X(-K_{\sF})\cong \det(\sF)$. A \emph{Fano foliation} $\sF\subsetneq T_X$ is a foliation such that $-K_{\sF}$ is an ample $\bQ$-Cartier divisor. We have the following two notions which generalise the Fano index of Fano varieties.

\begin{defn}
	Let $X$ be a normal projective variety and let $\sF\subsetneq T_X$ be a foliation such that $-K_{\sF}$ is a $\bQ$-Cartier divisor.
	\begin{enumerate}
		\item If $-K_{\sF}$ is big, the generalised index $\widehat{\iota}(\sF)$ of $\sF$ is defined as follows:
		\begin{center}
			$\widehat{\iota}(\sF)\coloneqq \sup\{t\in \bR\,|\,-K_{\sF} \equiv tH+P,$ where $H$ is an ample Cartier divisor and $P$ is a pseudoeffective $\bR$-Cartier $\bR$-divisor$\}$.
		\end{center}
		
		\item If $-K_{\sF}$ is ample, the Fano index $\iota(\sF)$ of $\sF$ is defined as the largest rational number such that $-K_{\sF} \sim_{\bQ} \iota(\sF) H$ for some ample Cartier divisor $H$.
	\end{enumerate}
\end{defn}

Kobayashi-Ochiai's theorem was successfully generalised to Fano foliations in the last decade, see \cite{AraujoDruelKovacs2008,AraujoDruel2013,AraujoDruel2014,Hoering2014,AraujoDruel2019} and the references therein. Our first theorem slightly generalises these results, see also \cite[Theorem 1.5]{AraujoDruel2019}.

\begin{thm}
	\label{t.Kobayashi-Ochiai-I}
	Let $X$ be a $\bQ$-factorial normal projective variety and let $\sF\subsetneq T_X$ be a foliation on $X$ such that $-K_{\sF}$ is a big $\bQ$-Cartier divisor. Then $r^a\geq \widehat{\iota}(\sF)$.
\end{thm}

As an application, we obtain a Kobayashi-Ochiai theorem for foliations. See Example \ref{e.generalised-cone} for the notion of normal generalised cone and see also \cite[Theorem 1.1]{AraujoDruelKovacs2008}, \cite[Theorem 1.2]{AraujoDruel2014}, \cite[Theorem 1.3]{Hoering2014} and \cite[Corollary 1.6]{AraujoDruel2019} for related results.

\begin{thm}
	\label{t.Kobayashi-Ochiai-II}
	Let $X$ be a $\bQ$-factorial normal projective variety and let $\sF\subsetneq T_X$ be a Fano foliation on $X$. Then $r^a\geq \iota(\sF)$ and the equality holds if and only if $X$ is a normal generalised cone over a polarised $\bQ$-factorial variety $(T,\sL)$ with vertex $\bP^{r^a-1}$ and there exists a foliation $\sG$ on $T$ with $K_{\sG}\sim_{\bQ} 0$ such that $\sF$ is the pull-back of $\sG$ under the natural rational map $h:X\dashrightarrow T$.
\end{thm}

Next we study the local positivity of the anti-canonical divisor of a foliation. Recall that Demailly introduced in \cite{Demailly1992} the notion of \emph{Seshadri constant} to measure the local positivity of nef divisors on projective varieties at a smooth point.

\begin{defn}
	Let $D$ be a nef $\bR$-Cartier $\bR$-divisor on a normal projective variety $X$ and $x\in X$ a smooth point. The Seshadri constant of $D$ at $x$ is defined as
	\[
	\epsilon(D,x)\coloneqq\inf\frac{D\cdot C}{\mult_x C},
	\]
	where the infimum is taken over all irreducible curves $C$ passing through $x$. We denote by $\epsilon(L)$ to be the supremum of $\epsilon(L,x)$ as $x$ varies over all smooth points of $X$.
\end{defn}

According to the lower semi-continuity of the function $\epsilon(D,\cdot)$, the supremum $\epsilon(L)$ is attained at a very general point of $X$ (see \cite[Example 5.1.11]{Lazarsfeld2004} or Lemma \ref{l.properties-Seshadri}). The Seshadri constant has many interesting properties. For example, we have the following theorem proved by Zhuang in \cite{Zhuang2018a}, which can be regarded as the local analogue of Kobayashi-Ochiai's theorem, see also \cite{BauerSzemberg2009} and \cite{LiuZhuang2018} for related results.

\begin{thm}[\protect{\cite[Theorem 1.5  and Theorem 1.7]{Zhuang2018a}}]
	\label{t.Zhuang-Fano-Seshadri}
	Let $X$ be an $n$-dimensional normal projective variety such that $-K_X$ is a nef and big $\bQ$-Cartier divisor. Then
	\begin{enumerate}
		\item $\epsilon(-K_X)\leq n+1$, and
		
		\item $X\cong \bP^n$ if $\epsilon(-K_X)>n$, and
		
		\item $X$ is rationally connected if $\epsilon(-K_X)> n-1$.
	\end{enumerate}
\end{thm}

In the following we aim to generalise Theorem \ref{t.Zhuang-Fano-Seshadri} to foliations and it can be viewed as the local counterpart of Kobayashi-Ochiai's theorem for foliations, see \cite[Theorem 1.5 and Corollary 1.6]{AraujoDruel2019} for related results.

\begin{thm}
	\label{c.upper-bound-Seshadri-weak-Fano}
	Let $X$ be a $\bQ$-factorial normal projective variety and let $\sF\subsetneq T_X$ be a foliation with $-K_{\sF}$ nef. Then $r^a\geq \epsilon(-K_{\sF})$.
\end{thm}
We refer the reader to Theorem \ref{t.bounding-Seshadri} for a more general statement. In the following theorem we address the rationally connectedness of general leaves of the algebraic part $\sF^a$ of a Fano foliation $\sF$ with Seshadri constant $\epsilon(-K_{\sF})>r^a-1$, see \cite[Theorem 1.8 and Corollary 1.9]{AraujoDruel2019} and Section \ref{s.RC-Seshadri} for related results and further discussion.

\begin{thm}
	\label{c.RC-Folia-large-Seshadri}
	Let $X$ be a $\bQ$-factorial normal projective variety and let $\sF\subsetneq T_X$ be a foliation with $-K_{\sF}$ nef and big. If $\epsilon(-K_{\sF}) > r^a - 1$, then the closure of a general leaf of the algebraic part $\sF^a$ of $\sF$ is rationally connected.
\end{thm}

We refer the reader to Theorem \ref{t.rationally-connectedness} for a more general statement. In the viewpoint of Kobayashi-Ochiai's theorem for Fano foliations, it is natural to ask if it is possible to classify those Fano foliations such that the Seshadri constants of their anti-canonical divisors are equal to their algebraic ranks. In the following we give a full list for such foliations on smooth projective varieties, see \cite[Corollary 1.6]{AraujoDruel2019}.

\begin{thm}
	\label{t.Seshadri-MaxValue}
	Let $\sF\subsetneq T_X$ be a Fano foliation on an $n$-dimensional projective manifold $X$. Then $r^a=\epsilon(-K_{\sF})$ if and only if $X\cong \bP^n$ and $\sF$ is the linear pull-back of a purely transcendental foliation on $\bP^{n-r^a}$ with zero canonical class.
\end{thm}

The sharpness of the results above will be discussed in Section \ref{s.Examples-Seshadri}. In particular, we will construct several examples to illustrate the possibilities of the values of $\widehat{\iota}(\sF)$, $\iota(\sF)$ and $\epsilon(-K_{\sF})$ in general setting. The main result of Section \ref{s.Examples-Seshadri} can be read as follows.

\begin{prop}
	\label{p.existence-examples}
	Let $(r,n)$ be a pair of integers such that $0< r<n$. 
	\begin{enumerate}
		\item\label{i.generalised-index} If $n\geq 3$, for any rational number $0< c\leq r$, there exists an $n$-dimensional projective manifold $X$ and a foliation $\sF\subsetneq T_X$ with algebraic rank $r^a=r$ such that $-K_{\sF}$ is big and $\widehat{\iota}(\sF)=c$.
		
		\item\label{i.Fano-index} If $n\geq 3$, for any rational number $0<c\leq \min\{r,n-2\}$, there exists an $n$-dimensional $\bQ$-factorial normal projective variety $X$ with klt singularities and a Fano foliation $\sF\subsetneq T_X$ with algebraic rank $r^a=r$ such that $\iota(\sF) = \widehat{\iota}(\sF) = c$.
		
		\item\label{i.Seshadri-constant} If $n\geq 2$, for any rational number $0<c\leq r$, there exists an $n$-dimensional $\bQ$-factorial normal projective variety $X$ with klt singularities and a Fano foliation $\sF\subsetneq T_X$ with algebraic rank $r^a=r$ such that $\epsilon(-K_{\sF})= c$.
	\end{enumerate}
\end{prop}

Our statement is actually a bit more general. For $n=2$, we will construct some foliations $\sF$ on surfaces in Example \ref{e.curve-foliations} to show that the rational numbers $c$ contained in the set $\{1-1/a\,|\,a\in\bZ_{\geq 1}\}$ can be realised as the generalised index of $\sF$ as in \ref{i.generalised-index}. In particular, this allows us to give a positive answer to a question proposed by Araujo and Druel on the generalised indices of foliations on projective manifolds (see Example \ref{e.curve-foliations}) and we refer the reader to \cite[Question 4.4]{AraujoDruel2019} or Question \ref{q.AD-question} for a precise statement. For $r=n-1$ and $n\geq 2$, we will construct some codimension one algebraically integrable foliations $\sF$ in Example \ref{e.codimension-foliation} to show that the rational numbers $c$ contained in the set $\{n-2+1/a\,|\,a\in\bZ_{\geq 1}\}$ can be realised as the Fano index of $\sF$ as in \ref{i.Fano-index}.

This paper is organised as follows. In Section \ref{s.Seshadri-Fano} we collect some results concerning Seshadri constant and Fano varieties with large indices or large Seshadri constants. In Section \ref{s.foliations} we introduce basic notions on foliations. In Section \ref{s.bounding-algebraic-rank} we study the positivity of the anti-canonical divisors of foliations and then apply it to prove Theorem \ref{t.Kobayashi-Ochiai-I}, Theorem \ref{t.Kobayashi-Ochiai-II}, Theorem \ref{c.upper-bound-Seshadri-weak-Fano} and Theorem \ref{c.RC-Folia-large-Seshadri}. In Section \ref{s.max-Seshadri} we prove Theorem \ref{t.Seshadri-MaxValue}. In Section \ref{s.Examples-Seshadri} we exhibit some examples and propose a few interesting questions. In particular we prove Proposition \ref{p.existence-examples}.

\subsection*{Acknowledgements} 

I am grateful to St{\'e}phane Druel and Andreas H{\"o}ring for their helpful comments. This work is supported by the National Key Research and Development Program of China (No. 2021YFA1002300), the NSFC grants (No. 12001521 and No. 12288201) and the CAS Project for Young Scientists in Basic Research (No. YSBR-033). I would like to thank the anonymous referee for his/her very detailed report which helps me to correct numerous inaccuracies and also to improve the exposition of the paper.
	
\section{Seshadri constant and Fano varieties}
\label{s.Seshadri-Fano}

Throughout this paper we work over the field of complex numbers $\bC$. We will frequently use the terminology and results of the minimal model program (MMP) as explained in \cite{KollarMori1998}. We refer to Lazarsfeld’s book \cite{Lazarsfeld2004} for notions of positivity of $\bR$-divisors, in particular \cite[\S\,5]{Lazarsfeld2004} for a general discussion on Seshadri constant.

\subsection{Basic properties of Seshadri constant}

In this subsection we briefly recall some basic properties of Seshadri constant.

\begin{lem}
	\label{l.properties-Seshadri}
	Suppose that $D$ is a nef $\bR$-Cartier $\bR$-divisor on a normal projective variety $X$.
	\begin{enumerate}
		\item\label{i.bigness-Seshadri} $D$ is big if and only if $\epsilon(D)>0$.
		
		\item\label{i.max-value-Seshadri} For a very general point $x\in X$, we have $\epsilon(D)=\epsilon(D,x)$.
		
		\item\label{i.pull-back-Seshadri} If $\pi:Y\rightarrow X$ is a birational morphism and $y\in Y$ is a smooth point such that $\pi$ is an isomorphism in a neighbourhood of $y$, then we have 
		\[
		\epsilon(\pi^*D,y)=\epsilon(D,\pi(y)).
		\]
		
		\item\label{i.restriction-Seshadri} Let $Y\subset X$ be a closed subvariety and $y\in Y$ be a smooth point of both $Y$ and $X$. Then we have
		\[
		\epsilon(D,y) \leq \epsilon(D|_Y,y).
		\]
	\end{enumerate}
\end{lem}

\begin{proof}
	The statement \ref{i.bigness-Seshadri} follows from \cite[Proposition 5.1.9]{Lazarsfeld2004} and the statement \ref{i.max-value-Seshadri} follows from the lower semi-continuity of the Seshadri function $\epsilon(D,\cdot):X\rightarrow \bR_{\geq 0}$, see for instance \cite[Example 5.1.11]{Lazarsfeld2004}. The statements \ref{i.pull-back-Seshadri} and \ref{i.restriction-Seshadri} are direct consequences of the definition.
\end{proof}

\subsection{Negativity lemma}

Let $Y\rightarrow Z$ be a projective morphism of varieties and $D$ a $\bR$-Cartier $\bR$-divisor on $Y$. We say that $D$ is \emph{nef on the very general curves of $Y/Z$} if there is a countable union of proper closed subsets $W$ of $Y$ such that $D\cdot C\geq 0$ for any curve $C$ on $Y$ contracted over $Z$ satisfying $C\not\subseteq W$. We need the following general negativity lemma. 

\begin{lem}[\protect{Negativity Lemma I, \cite[Lemma 3.3]{Birkar2012}}]
	\label{l.negativity}
	Let $f:Y\rightarrow Z$ be a birational contraction and let $D$ be a $\bR$-Cartier $\bR$-divisor on $Y$ written as $D=D^+ - D^-$ with $D^+$, $D^-\geq 0$ having no common components. Assume that $D^-$ is $f$-exceptional and for each irreducible component $S$ of $D^-$, the restriction $-D|_S$ is nef on the very general curves of $S/Z$. Then $D^{-}=0$. In other words, the $\bR$-divisor $D$ is effective.
\end{lem}

\begin{lem}[\protect{Negativity Lemma II}]
	\label{l.generic-negativity}
	Let $g:Y\rightarrow Z$ be a fibration between normal projective varieties and $f:Y\rightarrow X$ a birational contraction to a $\bQ$-factorial normal projective variety $X$. Let $D$ be a $\bR$-Cartier $\bR$-divisor on $Y$ such that there exists an effective $\bR$-Cartier $\bR$-divisor $N$ such that $N$ does not dominate $Z$ and $D-N$ is movable. Let $E^+$ and $E^{-}$ be the $f$-exceptional effective $\bR$-divisors with no common components such that $D+E^+-E^-=f^*(f_*D)$. Then $E^-$ does not dominant $Z$.
\end{lem}

\begin{proof}
	There exist $f$-exceptional effective $\bR$-divisors $E_N^+$ and $E_N^-$ having no common components such that 
	\[
	N + E_N^+ -E_N^-= f^*(f_* N). 
	\]
	As $N$ does not dominate $Z$ and $f^*(f_*N)$ is effective, the effective $\bR$-divisor $E_N^-$ does not dominate $Z$. On the other hand, as $P\coloneqq D-N$ is movable, for any $f$-exceptional prime divisor $S$, the restriction $P|_S$ is pseudoeffective and hence $(P-f^*(f_*P))|_S$ is nef on the very general curves of $S/X$. Thus, by Lemma \ref{l.negativity}, there exists an effective $f$-exceptional $\bR$-divisor $E_P$ such that $P+E_P=f^*(f_*P)$. As a consequence, we have
	\[
	E^+ - E^{-} = f^*(f_*D) - D = f^*(f_*N) -N + f^*(f_*P) - P = E_N^+ - E_N^{-} +  E_P.
	\]
	This yields $E^-\leq E_N^-$ and hence $E^-$ does not dominate $Z$.
\end{proof}

\subsection{Weak log Fano varieties}

Firstly let us recall the Kobayashi-Ochiai's theorem in the singular setting which was proved by Araujo and Druel in \cite{AraujoDruel2014}. Recall that a \emph{pair} $(X,\Delta)$ consists of a normal projective variety $X$ and an effective $\bR$-divisor $\Delta$ such that $K_X+\Delta$ is $\bR$-Cartier.

\begin{thm}[\protect{\cite[Theorem 1.1]{AraujoDruel2014}}]
	\label{t.Kobayashi-Ochiai}
	Let $(X,\Delta)$ be an $n$-dimensional pair such that $-(K_X+\Delta)\sim_{\bQ} \iota A$ for an ample Cartier divisor $A$ on $X$ and $\iota\in \bQ$.
	\begin{enumerate}
		\item If $\iota>n$, then $n<\iota\leq n+1$, $X\cong \bP^n$ and $\deg(\Delta)<n+1-\iota$.
		
		\item If $\iota=n$, then either $X\cong \bP^n$ and $\deg(\Delta)=1$, or $\Delta=0$ and $X$ is isomorphic to a (possibly singular) quadric hypersurface in $\bP^{n+1}$.
	\end{enumerate}
	In particular, if $\iota\geq n$, then $(X,\Delta)$ is klt unless $(X,\Delta)\cong (\bP^n,H)$, where $H$ is a hyperplane in $\bP^n$.
\end{thm}

\begin{proof}
	Note that $\Delta$ is a $\bQ$-divisor, according to \cite[Theorem 1.1]{AraujoDruel2014}, it remains to prove the last statement. If $X$ is isomorphic to a normal quadric hypersurface, then it is well-known that $X$ is isomorphic to a cone over a smooth quadric and hence it has only klt singularities (cf. Example \ref{e.generalised-cone}). On the other hand, if $X\cong \bP^n$, by \cite[Lemma 2.2]{AraujoDruel2014}, the pair $(\bP^n,\Delta)$ with $\deg(\Delta)\leq 1$ is klt unless $\Delta$ is a hyperplane in $\bP^n$.
\end{proof}

Next we collect some results from \cite{Zhuang2018a} on weak log Fano varieties $(X,\Delta)$ such that the anti-log canonical divisor $-(K_X+\Delta)$ has large Seshadri constant. Recall that a birational map $g:X\dashrightarrow Z$ is called a \emph{contraction} if the inverse map $g^{-1}:Z\dashrightarrow X$ does not contract any divisors in $Z$. We need the following simple observation.

\begin{lem}
	\label{l.factorisation-maps}
	Let $h:X\rightarrow Y$ be a birational morphism between normal projective varieties. Assume that there exists a birational contraction $g:X\dashrightarrow Z$ to a normal projective variety $Z$ such that $g^*A = h^*D$, where $A$ is an ample $\bR$-Cartier $\bR$-divisor on $Z$ and $D$ is a nef $\bR$-Cartier  $\bR$-divisor on $Y$. Then $g$ factors through $h$, i.e., there exists a birational morphism $f:Y\rightarrow Z$ such that $g=f\circ h$.
\end{lem}

\begin{proof}
	Let $\mu:W\rightarrow X$ be a resolution of $g:X\dashrightarrow Z$ and denote by $\nu\coloneqq g\circ \mu: W \rightarrow Z$ the induced birational morphism. As $g^*A = h^*D$, there exists a unique $\mu$-exceptional divisor $E$ such that
	\[
	\nu^* A = \mu^* h^*D + E.
	\]
	As $A$ is ample and $E$ is $(h\circ \mu)$-exceptional, by the negativity lemma (cf. Lemma \ref{l.negativity}), we get $-E\geq 0$. On the other hand, as $D$ is nef and $E$ is $\nu$-exceptional, it follows from the negativity lemma again that $E\geq 0$ and hence $E=0$. In particular, as $A$ is ample and $D$ is nef, every curve contracted by $h\circ\mu:W\rightarrow Y$ is contracted by $\nu:W\rightarrow Z$. Thus, by rigidity result \cite[Lemma 1.15]{Debarre2001}, the morphism $\nu$ factors through $h\circ \mu$, i.e. there exists a birational morphism $f:Y\rightarrow Z$ such that $f\circ h \circ \mu = \nu = g\circ \mu$.
\end{proof}

\begin{thm}[\protect{\cite[Theorem 1.5 and Lemma 5.6]{Zhuang2018a}}]
	\label{t.Zhuang}
	Let $(X,\Delta)$ be an $n$-dimensional pair such that $A\coloneqq -(K_X+\Delta)$ is nef.
	\begin{enumerate}
		\item If $\epsilon(A)>n-1$, then $X$ is rationally connected.
		
		\item If $\epsilon(A)\geq n$ and $(X,\Delta)$ is not klt, then $\epsilon(A)=n$ and there exists a birational morphism $f:X\rightarrow \bP^n$ with a hyperplane $H$ in $\bP^n$ such that 
		\begin{center}
			$\Delta=f_*^{-1}H$\quad and\quad $K_X+\Delta=f^*(K_{\bP^n}+H)$.
		\end{center}
	\end{enumerate}
\end{thm}

\begin{proof}
	The first statement follows from \cite[Theorem 1.5]{Zhuang2018a}. Here we note that the boundary $\Delta$ is assumed to be an effective $\bQ$-divisor in \cite[Theorem 1.5]{Zhuang2018a}, however the proof still works for $\Delta$ being an effective $\bR$-divisor. To prove the second statement, note that it follows from \cite[Lemma 5.6]{Zhuang2018a} that such a birational map $f:X\dashrightarrow \bP^n$ exists as a rational map. We remark again that the proof given there also works for the boundary $\Delta$ being a $\bR$-divisor. Since $-(K_{\bP^n}+H)$ is ample, applying Lemma \ref{l.factorisation-maps} to $\Id:X\rightarrow X$ and $f:X\dashrightarrow \bP^n$ yields that $f$ is actually a morphism.
\end{proof}

\section{Foliations}
\label{s.foliations}

In this section we collect and recall basic facts concerning foliations.

\subsection{Basic notions of foliations}

Let $\sF$ be a coherent sheaf over a normal variety $X$. Then \emph{rank} $r$ of $\sF$ is defined to be its rank at a general point of $X$. The dual sheaf $\sHom(\sF,\sO_X)$ will be denoted by $\sF^*$. The \emph{reflexive hull} of $\sF$ is defined as $\sF^{**}$ and $\sF$ is called reflexive if $\sF=\sF^{**}$. Given a positive integer $m$, we denote by $\wedge^{[m]}\sF$ the reflexive sheaf $(\wedge^m\sF)^{**}$ and by $\otimes^{[m]}\sF$ the reflexive sheaf $(\otimes^m\sF)^{**}$. In particular, the \emph{determinant} $\det(\sF)$ is defined as $\wedge^{[r]}\sF$. If $\pi:Y\rightarrow X$ is a morphism of varieties, then we write $\pi^{[*]}\sF$ for $(\pi^*\sF)^{**}$. If $X\rightarrow B$ is a morphism, we denote by $\Omega_{X/B}$ the relative K{\"a}hler differential and by $T_{X/B}$ the dual sheaf $\Omega_{X/B}^*$. Moreover, for simplicity, we will write $\Omega^r_{X/B}$ (resp. $\Omega^{[r]}_{X/B}$) instead of $\wedge^r\Omega_{X/B}$ (resp. $\wedge^{[r]}\Omega_{X/B}$).

\begin{defn}
	Let $X$ be a normal variety. A foliation on $X$ is a non-zero coherent subsheaf $\sF\subsetneq T_X$ satisfying
	\begin{enumerate}
		\item $\sF$ is saturated in $T_X$, i.e. $T_X/\sF$ is torsion free, and
		
		\item $\sF$ is closed under the Lie bracket.
	\end{enumerate}
    The canonical class $K_{\sF}$ of $\sF$ is any Weil divisor on $X$ such that $\sO_X(-K_{\sF})\cong \det(\sF)$.
\end{defn}

	Let $\sF\subsetneq T_X$ be a foliation of rank $r$ on a normal variety $X$. Then we have a natural morphism $\Omega^{r}_X\rightarrow \det(\sF^*)\cong \sO_X(K_{\sF})$, which induces a morphism 
	\[
	\eta:(\Omega^r_X\otimes \sO_X(-K_{\sF}))^{**} \rightarrow \sO_X.
	\]
	The \emph{singular locus} $\sing(\sF)$ of $\sF$ is defined to be the closed subscheme of $X$ whose ideal sheaf $\sI_S$ is the image of $\eta$, see \cite[Definitnion 3.4]{AraujoDruel2014} for more details. In particular, if $\sing(\sF)=\emptyset$, then $\sF$ is said to be \emph{regular}.
	
	Let $\varphi:X\dashrightarrow Y$ be a dominant rational map with connected fibres between normal varieties. Let $X^o\subset X$ and $Y^o\subset Y$ be the smooth open subsets such that the restriction of $\varphi$ to $X^o$ induces a dominant morphism $\varphi|_{X^o}=\varphi^o:X^o\rightarrow Y^o$. Let $\sG$ be a foliation on $Y$. The \emph{pull-back of $\sG$ via $\varphi$} is defined as the unique foliation $\sF$ on $X$ such that $\sF|_{X^o}=(d\varphi^o)^{-1}(\sG|_{Y^o})$. In this case we write $\sF=\varphi^{-1}\sG$. 

\begin{defn}[\protect{\cite[Lemma 2.4]{LorayPereiraTouzet2018} and \cite[Definition 2.4]{AraujoDruel2019}}]
	\label{d.algebraic-part}
	Let $X$ be a normal variety and let $\sF\subsetneq T_X$ be a foliation on $X$. Then there exists a normal variety $Y$, unique up to birational equivalence, a foliation $\sG$ on $Y$ and a dominant rational map $\varphi: X\dashrightarrow Y$ with connected fibres satisfying the following conditions:
	\begin{enumerate}
		\item the foliation $\sG$ is purely transcendental, i.e., there is no algebraic subvariety with positive dimension through a general point of $Y$ that is tangent to $\sG$;
		
		\item the foliation $\sF$ is the pull-back of $\sG$ via $\varphi$. 
	\end{enumerate}
    The foliation $\sF^a$ on $X$ induced by the rational map $\varphi: X\dashrightarrow Y$ is called the algebraic part of $\sF$. The algebraic rank of $\sF$ is defined as the rank of $\sF^a$ and we will denote it by $r^a$. Moreover, if $r^a=r$, then we say that the foliation $\sF$ is algebraically integrable. The foliation $\sG$ on $Y$ is called the transcendental part of $\sF$.
\end{defn}

Let $X$ be a normal projective variety and let $\sF$ be an algebraically integrable foliation on $X$ of rank $r>0$ such that $K_{\sF}$ is $\bQ$-Cartier. Let $i:F\rightarrow X$ be the normalisation of the closure of a general leaf of $\sF$. By \cite[Definition 3.11]{AraujoDruel2014}, there is a canonically defined effective $\bQ$-divisor $\Delta_F$ on $F$ such that $K_F+\Delta_F\sim_{\bQ} i^*K_{\sF}$. This pair $(F,\Delta_F)$ is called a \emph{general log leaf} of $\sF$. Let $T'$ be the unique proper subvariety of the Chow variety of $X$ whose general point parametrises the closure of  a general leaf of $\sF$ (viewed as a reduced and irreducible cycle in $X$). Let $T$ be the normalisation of $T'$ and $U\rightarrow T'\times X$ the normalisation of the universal cycle, with induced morphisms:
\begin{equation}
	\label{e.family-leaves}
	\begin{tikzcd}[row sep=large, column sep=large]
		U \arrow[r,"\nu"] \arrow[d,"\pi" left]
		& X \\
		T
		&
	\end{tikzcd}
\end{equation}
Then $\nu:U\rightarrow X$ is birational and, for a general point $t\in T$, the image $\nu(\pi^{-1}(t))\subsetneq X$ is the closure of a leaf of $\sF$. We shall call the diagram \eqref{e.family-leaves} the \emph{family of leaves} of $\sF$, see \cite[Lemma 3.9]{AraujoDruel2014}. Let $m$ be the Cartier index of $K_{\sF}$. Thanks to \cite[Lemma 3.7 and Remark 3.12]{AraujoDruel2014}, the foliation $\sF$ induces a natural generically surjective morphism
\begin{equation}
	\label{e.Pfaffian-fields}
	\otimes^m\Omega_{U/T}^r \longrightarrow \nu^*\sO_X(mK_{\sF}).
\end{equation}
Let $\sF_U$ be the algebraically integrable foliation on $U$ induced by $\sF$, or equivalently by $\pi$. Then we have a natural isomorphism
\[
\Omega_{U/T}^{[r]} \longrightarrow \sO_{U}(K_{\sF_U}).
\]
As $\nu^{*}\sO_X(mK_{\sF})$ is invertible, after taking reflexive hull of $\otimes^m\Omega_{U/T}^r$, the morphism \eqref{e.Pfaffian-fields} yields
\begin{equation}
	\label{e.pull-back-Pfaffian}
	\otimes^m \Omega_{U/T}^r \rightarrow \otimes^{[m]}\Omega_{U/T}^r \rightarrow \otimes^{[m]}\Omega_{U/T}^{[r]} \xrightarrow{\cong} \sO_U(mK_{\sF_U}) \rightarrow \nu^*\sO_X(mK_{\sF}).
\end{equation}
In particular, there exists a canonically defined effective Weil $\bQ$-divisor $\Delta$ on $U$ such that
\[
K_{\sF_U} + \Delta \sim_{\bQ} \nu^*K_{\sF}.
\]
Then $\Delta$ is $\nu$-exceptional as $\nu_*K_{\sF_U}=K_{\sF}$. Moreover, for a general point $t\in T$, set $U_t\coloneqq \pi^{-1}(t)$ and $\Delta_t\coloneqq \Delta|_{U_t}$. Then $(U_t,\Delta_t)$ coincides with the general log leaf $(F,\Delta_F)$ defined above. Here we note that $\Delta$ is $\bQ$-Cartier along codimension one points of $U_t$ and thus the restriction $\Delta|_{U_t}$ is well-defined as $\bQ$-divisor.

\subsection{Log algebraic part of a general leaf}
\label{s.log-algebraic-part}

Let $X$ be a $\bQ$-factorial normal projective variety and let $\sF\subsetneq T_X$ be a foliation with algebraic rank $r^a>0$. Let $\pi:U\rightarrow T$ be the family of leaves of $\sF^a$. Let $\mu:Z\rightarrow T$ be a resolution. Since $\pi:U\rightarrow T$ is equidimensional and $T$ is normal, by Chevalley's criterion \cite[Corollaire 14.4.4]{Grothendieck1966}, the morphism $\pi$ is actually universally open. In particular, since the general fibres of $\pi$ are irreducible, it follows that the fibre product $U\times_T Z$ is again irreducible. Denote by $Y$ the normalisation of $U\times_T Z$. Let $\rho:Y\rightarrow X$ and $q:Y\rightarrow Z$ be the induced morphisms. Then $q$ is also universally open and we have the following commutative diagram:
\[
\begin{tikzcd}[column sep=large, row sep=large]
	Y \arrow[r] \arrow[rr, bend left, "\rho"] \arrow[d,"q" left]
	    & U \arrow[d,"\pi"] \arrow[r,"\nu"]
	        & X \\
	Z \arrow[r,"\mu"]
	    & T
	        & 
\end{tikzcd}
\]
Let $\sG$ be the unique foliation on $Z$ such that $q^{-1}\sG = \rho^{-1}\sF$; that is, the foliation $\sG$ is the transcendental part of $\rho^{-1}\sF$. Denote by $\sQ$ the reflexive hull of $\sF/\sF^a$. Then clearly we have $\rho_*q^*K_{\sG}=K_{\sQ}$, where $K_{\sQ}$ is a Weil divisor on $X$ such that $\sO_X(-K_{\sQ})\cong \det(\sQ)$. In particular, there exists a canonically defined $\rho$-exceptional $\bQ$-divisor $E_{\sG}$ on $Y$ such that
\[
q^*K_{\sG} + E_{\sG}= \rho^*(\rho_*q^*K_{\sG}) = \rho^*K_{\sQ}.
\]
Denote by $\sH$ the pull-back foliation $\rho^{-1}\sF^a$, or equivalently the algebraically integrable foliation induced by $q$. Let $E_{\sF^a}$ be the canonically defined $\rho$-exceptional effective $\bQ$-divisor on $Y$ such that 
\[
K_{\sH} + E_{\sF^a} \sim_{\bQ} \rho^*K_{\sF^a}.
\]
Then we obtain
\begin{equation}
	\label{e.Canonical-bundle}
	K_{\sH} + E_{\sF^a} + E_{\sG} \sim_{\bQ} \rho^*(K_{\sF} - K_{\sQ}) + E_{\sG} \sim_{\bQ} \rho^*K_{\sF} - q^*K_{\sG}.
\end{equation}

We need the following remarkable result of Campana and P{\u{a}}un and we refer to \cite[Theorem 1.1]{CampanaPaun2019} and Theorem \ref{t.CP-Thm} for more details.

\begin{thm}[\protect{\cite{CampanaPaun2019}}]
	\label{t-CP-simplified}
	Let $X$ be a normal projective variety and let $\sF\subsetneq T_X$ be a purely transcendental foliation such that $K_{\sF}$ is $\bQ$-Cartier. Then $K_{\sF}$ is pseudoeffective.
\end{thm}

\begin{lem}
	\label{l.negative-components}
	Write $E_{\sG} = E_{\sG}^+ - E_{\sG}^-$ such that $E_{\sG}^+$, $E_{\sG}^-\geq 0$ with no common components. Then $E_{\sG}^-$ does not dominate $Z$.
\end{lem}

\begin{proof}
	Since $\sG$ is purely transcendental, $K_{\sG}$ is pseudoeffective by Theorem \ref{t-CP-simplified}. Let $K_{\sG}=N+P$ be the divisorial Zariski decomposition \cite{Boucksom2004,Nakayama2004}. Since $q$ is equidimensional, the pull-back $q^*P$ is movable. Then the result follows from Lemma \ref{l.generic-negativity}.
\end{proof}

Let $F$ be a general fibre of $q$. Set $D_F=(E_{\sF^a}+E_{\sG})|_F$ and $\Delta_F=E_{\sF^a}|_F$. Then the pair $(F,\Delta_F)$ is nothing but the general log leaf of $\sF^a$ and  Lemma \ref{l.negative-components} above says that $D_F$ is an effective $\bQ$-divisor such that $D_F\geq \Delta_F$ and
\[
K_F + D_F \sim_{\bQ} (K_{\sH}+E_{\sF^a}+E_{\sG})|_F \sim_{\bQ} \rho^*K_{\sF}|_F.
\]
Moreover, one can easily see that $D_F=\Delta_F$ if and only if $K_{\sQ}|_F\equiv 0$; that is, the $\bQ$-Cartier divisor $K_{\sQ}$ is numerically trivial along the closure of general leaves of $\sF^a$. We will call the pair $(F,D_F)$ \emph{the log algebraic part of a general leaf} of $\sF$. Note that if $\sF$ is algebraically integrable, then we have $D_F=\Delta_F$ and the pair $(F,D_F)$ is exactly the log leaf of $\sF$.

\section{Bounding the algebraic rank}
\label{s.bounding-algebraic-rank}

In this section we will study the lower bounds of algebraic ranks of Fano foliations and the goal is to prove Theorem \ref{t.Kobayashi-Ochiai-I}, Theorem \ref{t.Kobayashi-Ochiai-II}, Theorem \ref{c.upper-bound-Seshadri-weak-Fano} and Theorem \ref{c.RC-Folia-large-Seshadri}. 

\subsection{Positivity of anti-canonical divisors of foliations}

We need the following theorem due to Araujo and Druel.

\begin{thm}[\protect{\cite[Theorem 5.1]{AraujoDruel2013}}]
	\label{t.AD-fibration}
	Let $X$ be a normal projective variety and let $f:X\rightarrow C$ be a surjective morphism with connected fibres onto a smooth curve. Let $\Delta^+$ and $\Delta^{-}$ be effective $\bQ$-divisors on $X$ with no common components such that $f_*\sO_{X}(k\Delta^-)=\sO_C$ for every nonnegative integer $k$. Set $\Delta=\Delta^+ - \Delta^-$ and assume that $K_X+\Delta$ is $\bQ$-Cartier. 
	\begin{enumerate}
		\item If $(X,\Delta)$ is klt over the generic point of $C$, then $-(K_{X/C}+\Delta)$ is not nef and big.
		
		\item If $(X,\Delta)$ is lc over the generic point of $C$, then $-(K_{X/C}+\Delta)$ is not ample.
	\end{enumerate}
\end{thm}

The following result is a variant of \cite[Proposition 3.14]{AraujoDruel2014} and the proof is essentially a combination of arguments and results due to Araujo and Druel from \cite{AraujoDruel2013,AraujoDruel2014}. See also \cite[Proposition 5.8]{AraujoDruel2013}, \cite[Proposition 2.12 and Proposition 4.6]{Druel2017b} and \cite[Theorem 4.1]{Liu2019}.

\begin{prop}
	\label{t.variant-AD}
	Let $\sF\subsetneq T_X$ be a foliation on a $\bQ$-factorial normal projective variety $X$ with algebraic rank $r^a>0$. Let $i:F\rightarrow i(F)\subset X$ be the normalisation of the closure of a general leaf of $\sF^a$. Let $P$ be an effective $\bQ$-divisor on $X$ and denote by $D_P$ the pull-back $i^*P$.
	\begin{enumerate}
		\item If the pair $(F,D_F+D_P)$ is klt, then $-(K_{\sF}+P)$ is not nef and big. 
		
		\item If the pair $(F,D_F+D_P)$ is lc and $-(K_{\sF}+P)$ is ample, then there is a common point in the closure of general leaves of $\sF^a$.
	\end{enumerate}
\end{prop}

\begin{proof}
	Throughout the proof we shall follow the notation in Section \ref{s.log-algebraic-part}. In particular, we have the following commutative diagram
	\[
	\begin{tikzcd}[column sep=large, row sep=large]
		Y \arrow[r,"h"] \arrow[rr, bend left, "\rho"] \arrow[d,"q" left]
		& U \arrow[d,"\pi"] \arrow[r,"\nu"]
		& X \\
		Z \arrow[r,"\mu"]
		& T
		& 
	\end{tikzcd}
	\]
	such that $Z$ is a smooth projective variety and the image of a general fibre of $q$ under $\rho$ is the closure of a general leaf of $\sF^a$. Denote by $\sH$ the foliation defined by $q$ and let $\sG$ be the transcendental part of $\rho^{-1}\sF$. 
	
	As explained in the beginning of Section \ref{s.log-algebraic-part}, the morphism $q:Y\rightarrow Z$ is universally open. Let $C$ be a general complete intersection curve in $X$ which is disjoint from the closed subset $\rho(\Exc(\rho))$. Then we can identify $C$ with $\rho^{-1}(C)\subset Y$. Let us denote $q(C)$ by $B$. Let $B'$ be the normalisation of $B$ and let $Y_{B'}$ be the fibre product $Y\times_Z B'$. Then we may assume that natural projection $Y_{B'}\rightarrow B'$ is open and its general fibres are connected and normal. In particular, the variety $Y_{B'}$ is irreducible and hence $(Y_{B'})_{\red}\rightarrow B'$ is flat. Thus, thanks to \cite[Theorem 2.1]{BoschLuetkebohmertRaynaud1995}, there exists a finite morphism $C'\rightarrow B'$ such that $q':Y' \rightarrow C'$ is flat with reduced fibres, where $Y'$ is the normalisation of $(Y_{B'})_{\red}\times_{B'} C'$ and $q':Y'\rightarrow C'$ is the morphism induced by the projection $(Y_{B'})_{\red}\times_{B'} C' \rightarrow C'$. 
	
	Let $\sH'$ be the algebraically integrable foliation on $Y'$ induced by the natural projection $q':Y'\rightarrow C'$. Denote by $g:Y'\rightarrow Y$ the induced finite morphism. Let $m$ be the Cartier index of $K_{\sF^a}$. Thanks to \cite[Remark 3.12]{AraujoDruel2014}, we have a generically surjective map
	\[
	\otimes^m \Omega_{Y'/C'}^{r^a}  \rightarrow g^*\rho^*\sO_{X}(mK_{\sF^a}).
	\]
	This implies that there exists a canonically defined effective $\bQ$-divisor $\Delta'$ on $Y'$ such that $K_{\sH'} + \Delta' \sim_{\bQ} g^*\rho^*K_{\sF^a}$. Moreover, let $F'$ be a general fibre of $q':Y'\rightarrow C'$. Since both $C$ and $F$ are general, we may assume that $g(F')=F$ and the pair $(F',\Delta'|_F)$ is isomorphic to the pair $(F,\Delta_F)$, which is the general log leaf of $\sF^a$. 
    
    On the other hand, since $K_{\sG}$ is pseudoeffective, we can assume that $K_{\sG}\cdot B\geq 0$ since $C$ is general. In particular, the pull-back $g^*q^*K_{\sG}$ is nef. Since the fibres of $q':Y'\rightarrow C'$ are reduced, we have $K_{\sH'}=K_{Y'/C'}$ and
    \begin{align*}
    	K_{Y'/C'}+\Delta'+g^*E_{\sG} = K_{\sH'} + \Delta' + g^*E_{\sG}  
    	    & \sim_{\bQ} g^*\rho^*K_{\sF^a} + g^*E_{\sG} \\
    	    & \sim_{\bQ} g^*\rho^*K_{\sF} - g^*q^*K_{\sG}.
    \end{align*}
    Let us denote by $D'$ the $\bQ$-divisor $\Delta'+g^*E_{\sG}$ and by $D_{F'}$ the restriction $D'|_{F'}$. Then the pair $(F',D'_F)$ is isomorphic to the pair $(F,D_F)$ and the pair $(F', D_{F'} + g^*\rho^*P|_{F'})$ is isomorphic to $(F,D_F+D_P)$.
    
    We write $D'=D'^+ - D'^-$ with $D'^+$, $D'^-\geq 0$ having no common components. Then clearly $g(\supp(D'^-))$ is contained in $\supp(E_{\sG}^-)$, which is contained in $\Exc(\rho)$. In particular, the curve $C$ is disjoint from $g(\supp(D'^-))$, which implies that there is no fibre of $q':Y'\rightarrow C'$ contained in $\supp(D'^-)$. Thus, we have $q'_*\sO_{Y'}(kD'^-)=\sO_{C'}$ for every non-negative integer $k$.
    
    Now we assume that $(F,D_F+D_P)$ is klt. Then the pair $(Y',D'+g^*\rho^*P)$ has klt singularities over the generic point of $C'$ by inversion of adjunction. Applying Theorem \ref{t.AD-fibration} yields that the $\bR$-Cartier $\bR$-divisor
    \[
    -(K_{Y'/C'}+D'+g^*\rho^*P) = -(K_{\sH'}+D'+g^*\rho^*P) \sim_{\bR} -g^*\rho^*(K_{\sF}+P) + g^*q^*K_{\sG}
    \]
    cannot be nef and big. As $g^*q^*K_{\sG}$ is nef, one see that $-(K_{\sF}+P)$ cannot be nef and big as $C$ is general and the first statement follows.
    
    Finally we assume that $(F,D_F+D_P)$ is lc and $-(K_{\sF}+P)$ is ample. Then the pair $(Y',D'+g^*\rho^*P)$ has lc singularities over the generic point of $C'$ by inversion of adjunction. Suppose to the contrary that there is no common point in the closure of the general leaves of $\sF^a$. Following the same argument as in \cite[Proposition 5.3]{AraujoDruel2013}, we can assume that the morphism $Y'\rightarrow (Y_B)_{\red}\rightarrow \rho((Y_B)_{\red})$ is finite and hence $-g^*\rho^*(K_{\sF}+P)$ is ample. In particular, since $g^*q^*K_{\sG}$ is nef, the anti-log canonical divisor $-(K_{Y'/C'}+D'+g^*\rho^*P)$ is ample, which contradicts Theorem \ref{t.AD-fibration}.
\end{proof}

\subsection{Kobayashi-Ochiai's theorem for foliations}

In this subsection we apply Proposition \ref{t.variant-AD} to prove a Kobayashi-Ochiai's theorem for foliations. We start with the following example which will be frequently used in Section \ref{s.Examples-Seshadri}.

\begin{example}
	\label{e.generalised-index}
	Let $Z$ be normal projective variety and let $\sO_Z(1)$ be an ample line bundle on $Z$. Given positive integers $r'$, $m$ and non-negative integers $b_1\geq \dots\geq b_{r'}\geq 0$, let us denote by $\sE$ the vector bundle 
	\[
	\sO_Z(m)\oplus \bigoplus_{i=1}^{r'}\sO_Z(-b_i).
	\]
	Set $b=\sum_{i=1}^{r'} b_i$ and denote by $X$ the projective bundle $\bP(\sE)$ with $\pi:X\rightarrow Z$ the natural projection. Let $\Lambda$ be the tautological divisor of $\bP(\sE)$ and let $A$ be a Cartier divisor on $Z$ such that $\sO_Z(A)\cong \sO_Z(1)$. Denoter by $\bP(\sQ)=E\subsetneq X$ the prime divisor associated to the quotient $\sE\rightarrow \oplus\sO_Z(-b_i)=\sQ$. Then we have $E\sim \Lambda - m\pi^*A$. 
\end{example}

\begin{example}[\protect{Normal generalised cone}]
	\label{e.generalised-cone}
	In Example \ref{e.generalised-index}, set $b_i=0$ for every $1\leq i\leq r'$. For an integer $e\gg 1$, the linear system $|\sO_{\bP(\sE)}(e)|$ induces a birational morphism $\mu:X\rightarrow Y$ onto a normal projective variety. The morphism $\mu$ contracts the divisor $E=\bP(\sO_Z^{\oplus r'})\subsetneq X$ onto $\mu(E)=\bP^{r'-1}$ and induces  an isomorphism $Y\setminus \mu(E)\cong X\setminus E$. We will call $X$ the \emph{normal generalised cone over the base $(Z,\sO_Z(m))$ with vertex $\mu(E)\cong\bP^{r'-1}$}. 
	
	By \cite[Remark 4.2]{AraujoDruel2014}, if $Z$ is $\bQ$-factorial and $\rho(Z)=1$, then so is $Y$. Moreover, by \cite[Example 3.8]{Kollar1997}, if $Z$ is klt and $-K_Z$ is ample, then $Y$ has only klt singularities and if $-K_Z\equiv 0$ and $Z$ is lc, then $Y$ has only lc singularities.
\end{example}

We recall the following general definition, see \cite[Lemma 4.1]{AraujoDruel2019}.

\begin{defn} 
	\label{d.generalised-index}
	Let $X$ be a normal projective variety and let $D$ be a big $\bR$-Cartier $\bR$-divisor on $X$. We define
	\begin{center}
		$\widehat{\iota}(D)\coloneqq \sup\{t\in \bR\,|\,D \equiv tA+P,$ where $A$ is an ample Cartier divisor and $P$ is a pseudoeffective $\bR$-Cartier $\bR$-divisor$\}$.
	\end{center} 
\end{defn}
Note that we have $\widehat{\iota}(D)<\infty$ and there exists an ample Cartier divisor $H$ on $X$ and a pseudoeffective $\bR$-Cartier $\bR$-divisor $P$ such that $D\equiv \widehat{\iota}(D)H+P$, see \cite[Lemma 4.1]{AraujoDruel2019}. Let $\sF\subsetneq T_X$ be a foliation on a normal projective variety $X$ such that $-K_{\sF}$ is a big $\bQ$-Cartier divisor. Then the \emph{generalised index} $\widehat{\iota}(\sF)$ of $\sF$ is defined as $\widehat{\iota}(-K_{\sF})$.

\begin{proof}[Proof of Theorem \ref{t.Kobayashi-Ochiai-I}]
	Let $q:Y\rightarrow Z$ and $\rho:Y\rightarrow X$ be the morphism defined as in Section \ref{s.log-algebraic-part}. Arguing by contraction we suppose that $\widehat{\iota}(\sF)>r^a$. Let $0<\varepsilon\ll 1$ be a sufficiently small positive real number such that $r^a<\widehat{\iota}(\sF)-\varepsilon\in\bQ$. Then by assumption the divisor $-K_{\sF} - (\widehat{\iota}(\sF)-\varepsilon) H$ is a big $\bQ$-divisor, where $H$ is an ample Cartier divisor such that $-K_{\sF}-\widehat{\iota}(\sF) H$ is pseudoeffective. In particular, there exists an effective $\bQ$-divisor $P$ on $X$ such that 
	\[
	-K_{\sF} \sim_{\bQ} (\widehat{\iota}(\sF) - \varepsilon) H + P.
	\]
	Let $i:F\rightarrow X$ be the normalisation of a general leaf of $\sF^a$ and let $(F,D_F)$ be the log algebraic part of a general leaf of $\sF$. Denote by $D_P$ the pull-back $i^*P$. Then applying Theorem \ref{t.Kobayashi-Ochiai} to the pair $(F,D_F+D_P)$ shows that $(F,D_F+D_P)$ has only klt singularities and we get a contraction by Proposition \ref{t.variant-AD}.
\end{proof}

\begin{proof}[Proof of Theorem \ref{t.Kobayashi-Ochiai-II}]
	By definition, we have $\iota(\sF)\leq \widehat{\iota}(\sF)$. Hence, we must have $\iota(\sF)\leq r^a$ by Theorem \ref{t.Kobayashi-Ochiai-I}. Moreover, one can easily derive from Theorem \ref{t.Kobayashi-Ochiai} and the proof of Theorem \ref{t.Kobayashi-Ochiai-I} that if the equality $\iota(\sF)=r^a$ holds, then the log algebraic part $(F,D_F)$ of a general leaf of $\sF$ is isomorphic to $(\bP^{r^a},H)$, where $H$ is a hyperplane in $\bP^{r^a}$.
	
	Let $q:Y\rightarrow Z$ and $\rho:Y\rightarrow X$ be the morphisms defined in Section \ref{s.log-algebraic-part}. Then the images of the general fibres of $q$ under $\rho$ are the closure of the leaves of $\sF^a$. In particular, the general fibre $F$ of $q$ is isomorphic to $\bP^{r^a}$. Let $A$ be an ample Cartier divisor on $X$ such that $-K_{\sF}\sim_{\bQ} \iota(\sF) A$. Then $\sM\coloneqq\rho^*\sO_X(A)$ is a $q$-ample line bundle such that $\sM|_F\cong \sO_{\bP^{r^a}}(1)$. Thus the pair $(Y,\sM)$ is isomorphic to $(\bP_Z(\sE),\sO_{\bP(\sE)}(1))$ as varieties over $Z$ by \cite[Proposition 4.10]{AraujoDruel2014}, where $\sE\coloneqq q_*\sM$ is a nef vector bundle over $Z$ with rank $r^a+1$.
	
	On the other hand, note that there exists a purely transcendental foliation $\sG$ on the smooth projective variety $Z$ such that $K_{\sG}$ is pseudoeffective and $\rho^{-1}\sF=q^{-1}\sG$. Moreover, by \eqref{e.Canonical-bundle} there exists a canonically defined $\rho$-exceptional $\bQ$-divisor $E$ on $Y$ such that
	\[
	K_{\sH} + E \sim_{\bQ} \rho^*K_{\sF} - q^* K_{\sG},
	\]
	where $\sH$ is the foliation defined by $q$ and the restriction $E|_F=D_F$ of $E$ to a general fibre $F$ of $q$ is a hyperplane in $F\cong\bP^{r^a}$. In particular, there exists a unique $q$-horizontal prime divisor $E_h$ on $Y$ such that $E-E_h$ is $q$-vertical and hence there exists a $\bQ$-divisor $D$ on $Z$ such that $q^*D=E-E_h$. Let $\{H_i\}_{1\leq i\leq  n-1}$ be a collection of ample Cartier divisors on $X$ and let $C$ be a general complete intersection of general members of the linear systems $|m_1 H_1|, \dots, |m_{n-1} H_{n-1}|$ with $m_i\gg 1$. Then the preimage $\rho^{-1}(C)$ is disjoint from $\supp(E)$ and as a consequence the curve $B\coloneqq q(\rho^{-1}C)$ is disjoint from $\supp(D)$. Let $n:B'\rightarrow B$ be the normalisation of $B$ and let $q':Y_{B'}\rightarrow B'$ be the fibre product $Y\times_Z B'\cong \bP(n^*\sE)$. Let $\sH'$ be the foliation on $Y_{B'}$ defined by $q'$ and denote by $g:Y_{B'} \rightarrow Y$ the natural morphism. Then we have
	\begin{equation}
		\label{e.curve-base-change}
		K_{\sH'} + g^*E_h = g^*K_{\sH} + g^*E \sim_{\bQ} g^*(\rho^*K_{\sF}-q^*K_{\sG}) \sim_{\bQ} -r^a A'- g^*q^*K_{\sG},
	\end{equation}
	where $A'\coloneqq g^*\rho^*A$. Set $\sE'=n^*\sE$. Then we have
	\begin{equation}
		\label{e.projective-bundle-}
		\sO_{Y_{B'}}(K_{\sH'}) \cong \sO_{Y_{B'}}(K_{Y_{B'}/B'}) \cong \sO_{Y_{B'}}(-(r^a+1) A') \otimes q'^*\det(\sE').
	\end{equation}
	Combining \eqref{e.curve-base-change} and \eqref{e.projective-bundle-} yields that there exists a sufficiently divisible positive integer $m$ such that
	\begin{equation}
		\label{e.non-vanishing-over-curves}
		\sO_{Y_{B'}}(m g^*E_h) \cong \sO_{Y_{B'}}(mA')\otimes q'^*\det(\sE')^{\otimes -m} \otimes q'^*\sO_{B'} (n^*K_{\sG})^{\otimes -m}. 
	\end{equation}
	As $\sO_{Y_{B'}}(A') \cong \sO_{\bP(\sE')}(1)$ and $mg^*E_h$ is effective, the isomorphism \eqref{e.non-vanishing-over-curves}  means
	\[
	H^0(B',\Sym^m\sE'\otimes \det(\sE')^{\otimes -m} \otimes \sO_{B'}(m n^*K_{\sG})^{\otimes -1})\not=0.
	\]Then \cite[Lemma 4.11]{AraujoDruel2014} yields $\deg(n^*K_{\sG})\leq 0$. As $K_{\sG}$ is pseudoeffective and $C$ is a general complete intersection, we must have $K_{\sG}\cdot B'=0$. In particular, since $\rho^{-1}(C)$ is disjoint from $\supp(E_{\sG})\subset\supp(\Exc(\rho))$, we obtain
	\[
	K_{\sQ} \cdot C = \rho^*K_{\sQ} \cdot \rho^{-1}(C) = (q^*K_{\sG}+ E_{\sG}) \cdot \rho^{-1}(C) = q^*K_{\sG}\cdot \rho^{-1}(C)=0,
	\]
	where $\sQ$ is the reflexive hull of the quotient $\sF/\sF^a$ and $K_{\sQ}$ is a Weil divisor such that $\sO_X(-K_{\sQ})\cong \det(\sQ)$. Since $K_{\sQ}$ is pseudoeffective and $C$ is a general complete intersection, by \cite[Lemma 6.5]{Peternell1994} we obtain $K_{\sQ}\equiv 0$. As a consequence, the algebraically integrable foliation $\sF^a$ is a Fano foliation with $-K_{\sF^a}\equiv r^a A$ and its general log leaf $(F,\Delta_F)$ is isomorphic to the log algebraic part $(F,D_F)=(\bP^{r^a},H)$ of $\sF$ as $K_{\sQ}|_F\equiv 0$. Thus, applying \cite[Theorem 1.3]{Hoering2014} to $(X,\sF^a)$ shows that the $X$ is a normal generalised cone over a $\bQ$-factorial polarised variety $(T,\sL)$ with vertex $\bP^{r^a-1}$. Here we remark that though \cite[Theorem 1.3]{Hoering2014} is stated for $\bQ$-linear equivalence, the proof given there still works for numerical equivalence in our situation because the $\bQ$-linear equivalence is only used to derive the description of the general log leaves by applying \cite[Proposition 4.5]{AraujoDruel2014}, or equivalently Theorem \ref{t.Kobayashi-Ochiai}, in Step 1 of its proof, see \cite[p.2476, Proof of Theorem 1.3]{Hoering2014} for details. However, this description holds automatically for $(X,\sF^a)$ by the argument above.
	
	Finally, by abuse of notation we may still denote by $\sG$ the foliation on $T$ such that $h^{-1}\sG = \sF$, where $h:X\dashrightarrow T$ is the natural rational map. Let $\pi:U\coloneqq\bP(\sL\oplus \sO_T^{\oplus r^a})\rightarrow T$ be corresponding projective bundle (cf. Example \ref{e.generalised-cone}) and let $E$ be the exceptional divisor of $\mu:U\rightarrow X$. Then it is clear that we have 
	\[
	\mu_*\pi^*K_{\sG}=K_{\sQ} \sim_{\bQ} K_{\sF} - K_{\sF^a} \sim_{\bQ} 0.
	\]
	On the other hand, we also have $\pi^*K_{\sG} + aE \sim_{\bQ} \mu^*K_{\sQ}\sim_{\bQ} 0$ for some $a\geq 0$. Nevertheless, as the restriction $E|_F$ of $E$ to a general fibre $F\cong \bP^{r^a}$ of $\pi$ is a hyperplane, we must have $a=0$ and hence $K_{\sG}\sim_{\bQ} 0$.
\end{proof}

\subsection{Proof of Theorem \ref{c.upper-bound-Seshadri-weak-Fano} and Theorem \ref{c.RC-Folia-large-Seshadri}}

In this subsection we finish the proofs of Theorem \ref{c.upper-bound-Seshadri-weak-Fano} and Theorem \ref{c.RC-Folia-large-Seshadri} and the proofs are similar to that of Theorem \ref{t.Kobayashi-Ochiai-I} by applying Theorem \ref{t.Zhuang}. The following theorem is a slight generalisation of Theorem \ref{c.upper-bound-Seshadri-weak-Fano}.

\begin{thm}
	\label{t.bounding-Seshadri}
	Let $X$ be a $\bQ$-factorial normal projective variety and let $\sF\subsetneq T_X$ be a foliation such that $-K_{\sF}\equiv A + P$, where $A$ is a nef $\bR$-divisor and $P$ is a pseudoeffective $\bR$-divisor. Then $r^a\geq \epsilon(A)$. 
\end{thm}

\begin{proof}
	If $r^a=0$, then $\sF$ is purely transcendental and hence $K_{\sF}$ is pseudoeffective by Theorem \ref{t-CP-simplified}. In particular, the pseudoeffective $\bR$-divisor $-A\equiv K_{\sF}+P$ is anti-nef by assumption and so $A\equiv 0$ and $\epsilon(A)=0$.
	
	Now suppose that $r^a>0$ and $\epsilon(A)>r^a$. Then the $\bR$-divisor $A$ is nef and big. After replacing $A$ by $(1-\varepsilon)A$ and $P$ by $(1-\varepsilon)A + P$ for a sufficiently small number $\varepsilon>0$, we may assume that $P$ is actually a big $\bR$-divisor. In particular, up to $\bR$-linear equivalence, we may also assume that $P$ is an effective $\bR$-divisor.
    
    Let $i:F\rightarrow i(F)\subset X$ be the normalisation of the closure of a general leaf of $\sF^a$. Denote by $(F,D_F)$ the log algebraic part of a general leaf of $\sF$. As $F$ is general, we can assume that $F$ is not contained in the support of $P$. In particular, the pull-back $D_P\coloneqq i^*P$ is a well-defined effective $\bR$-divisor and we have
    \begin{equation}
    	K_F + D_F + D_P \sim_{\bR} i^*K_{\sF} + i^*P \equiv -i^*A.
    \end{equation}
    By assumption, we have $\epsilon(i^*A)>r^a=\dim(F)$. Then Theorem \ref{t.Zhuang} implies that the pair $(F,D_F+D_P)$ is klt, which contradicts Proposition \ref{t.variant-AD}.
\end{proof}

\begin{cor}
	\label{c.log-leaf-maxSeshadri}
	Let $X$ be a $\bQ$-factorial normal projective variety, and $\sF\subsetneq T_X$ a foliation with $-K_{\sF}$ ample. If $\epsilon(-K_{\sF})\geq r^a$, then $\epsilon(-K_{\sF})=r^a$ and the log algebraic part $(F,D_F)$ of a general leaf of $\sF$ is isomorphic to $(\bP^{r^a},H)$, where $H$ is a hyperplane of $\bP^{r^a}$. Moreover, there is a common point in the closure of a general leaf of $\sF^a$. 
\end{cor}

\begin{proof}
	The equality $\epsilon(-K_{\sF})=r^a$ follows from Theorem \ref{t.bounding-Seshadri}. On the other hand, by Theorem \ref{t-CP-simplified}, we have $r^a>0$ as $-K_{\sF}$ is ample. Let $(F,D_F)$ be the log algebraic part of a general leaf of $\sF$. Then we have
	\begin{equation}
		\epsilon(-K_F-D_F) = \epsilon(-i^*K_{\sF}) \geq \epsilon(-K_{\sF}) \geq r^a=\dim(F).
	\end{equation}
    By Proposition \ref{t.variant-AD}, the pair $(F,D_F)$ is not klt. Then Theorem \ref{t.Zhuang} implies that there exists a birational morphism $\nu:F\rightarrow \bP^{r^a}$ such that there exits a hyperplane $H$ in $\bP^{r^a}$ satisfying $\nu^{-1}_*H=D_F$ and $\nu^*(K_{\bP^{r^a}}+H)=K_F+D_F$. As $i:F\rightarrow i(F)\subset X$ is the normalisation and $-K_{\sF}$ is ample, the anti-log canonical divisor
    \[
    -(K_F+D_F) \sim_{\bQ} -i^*K_{\sF}
    \]
    is ample. Moreover, as $-(K_{\bP^{r^a}}+H)$ is ample, it follows that $\nu$ is finite and hence $\nu$ is an isomorphism. As a consequence, the pair $(F,D_F)$ is lc and the result follows from Proposition \ref{t.variant-AD}.
\end{proof}

\begin{proof}[Proof of Theorem \ref{c.upper-bound-Seshadri-weak-Fano}]
	It follows from Theorem \ref{t.bounding-Seshadri}.
\end{proof}

The following theorem is a slight generalisation of Theorem \ref{c.RC-Folia-large-Seshadri}.

\begin{thm}
	\label{t.rationally-connectedness}
	Let $X$ be a $\bQ$-factorial normal projective variety and let $\sF\subsetneq T_X$ be a foliation such that $-K_{\sF} \equiv A + P$, where $A$ is a nef and big $\bR$-divisor and $P$ is a pseudoeffective $\bR$-divisor. If $\epsilon(A) > r^a - 1$, then the closure of a general leaf of the algebraic part $\sF^a$ of $\sF$ is rationally connected.
\end{thm}

\begin{proof}
	Let $0<\varepsilon\ll 1$ be a sufficiently small positive real number. After replacing $A$ by $(1-\varepsilon)A$ and $P$ by $(1-\varepsilon)A+P$, we may assume that $P$ is an effective $\bR$-divisor. Let $i:F\rightarrow i(F)\subset X$ be the normalisation of the closure of a general leaf of the algebraic part of $\sF$ and set $D_P\coloneqq i^*P$. Then we have
	\[
	K_F+D_F+D_P\sim_{\bR} i^*K_{\sF} + i^*P \equiv -i^*A,
	\]
	where $(F,D_F)$ is the log algebraic part of a general leaf of $\sF$. Since $\varepsilon$ is a small positive real number and $F$ is general, by assumption we still have $\epsilon(i^*A)>r^a-1=\dim(F)-1$ after replacing $A$ by $(1-\varepsilon)A$. Hence, it follows from Theorem \ref{t.Zhuang} that $F$ is rationally connected.
\end{proof}

\begin{proof}[Proof of Theorem \ref{c.RC-Folia-large-Seshadri}]
	It follows from Theorem \ref{t.rationally-connectedness}.
\end{proof}

\section{Fano foliations with maximal Seshadri constants}
\label{s.max-Seshadri}

In the section we study Fano foliations $\sF$ on smooth projective varieties $X$ such that $\epsilon(-K_{\sF})=r^a$ and in particular we will prove Theorem \ref{t.Seshadri-MaxValue}.

\subsection{Stability condition with respect to movable curve classes}

In this subsection we briefly recall some basic facts about stability of coherent sheaves with respect to a movable curve class, see \cite{GrebKebekusPeternell2016}. Given a normal projective variety $X$, we denote by $N_1(X)_{\bR}$ the space of numerical curve classes. A curve class $\alpha\in N_1(X)_{\bR}$ is called \emph{movable} if $D\cdot \alpha\geq 0$ for all effective Cartier divisors $D$ on $X$. 

Let $X$ be a $\bQ$-factorial normal projective variety and let $\alpha\in N_1(X)_{\bR}$ be a movable curve class. Similar to the classical case, given a torsion free coherent sheaf $\sF$ with positive rank on $X$, we can define the \emph{slope of $\sF$ with respect to $\alpha$} to be the real number
	\[
	\mu_{\alpha}(\sF) \coloneqq \frac{\det(\sF)\cdot \alpha}{\rank(\sF)}.
	\]
\begin{defn}
	Let $\sF$ be a non-zero torsion free coherent sheaf on a $\bQ$-factorial normal projective variety $X$ and let $\alpha\in N_1(X)_{\bR}$ be a movable curve class. The sheaf $\sF$ is $\alpha$-semistable (resp. $\alpha$-stable) if, for any subsheaf $\sE$ of $\sF$ such that $0<\rank(\sE)<\rank(\sF)$, one has
	\begin{center}
		$\mu_{\alpha}(\sE)\leq \mu_{\alpha}(\sF)$ (resp. $\mu_{\alpha}(\sE)<\mu_{\alpha}(\sF)$).
	\end{center}
\end{defn}   

A number of known results from the classical case are extended to this setting. For example, the existences of maximally destabilising subsheaf and Harder-Narasimhan filtration are proved in \cite{GrebKebekusPeternell2016}. More precisely, recall that the \emph{maximal and minimal slopes of $\sF$ with respect to $\alpha$} are defined as
\begin{center}
	$\mu_{\alpha}^{\max}(\sF)\coloneqq\sup\{\mu_{\alpha}(\sE)\,|\,0\not=\sE\subset \sF$ is a coherent subsheaf $\}$
\end{center}
and
\begin{center}
	$\mu_{\alpha}^{\min}(\sF)\coloneqq\inf\{\mu_{\alpha}(\sQ)\,|\,\sQ\not=0$ is a torsion-free quotient of $\sF\}$.
\end{center} 
By \cite[Proposition 2.22 and Corollary 2.24]{GrebKebekusPeternell2016}, there exists a unique non-zero coherent subsheaf $\sE$ of $\sF$ such that $\mu_{\alpha}(\sE)=\mu_{\alpha}^{\max}(\sF)$ and if $\sE'\subset \sF$ is any subsheaf with $\mu_{\alpha}(\sE')=\mu_{\alpha}^{\max}(\sF)$, then $\sE'\subset \sE$. We call this subsheaf $\sE$ the \emph{maximal destablising subsheaf of $\sF$ (with respect to $\alpha$)}. Moreover, by \cite[Corollary 2.26]{GrebKebekusPeternell2016}, there exists a unique \emph{Harder-Narasimhan filtration} of $\sF$; that is, a filtration
\[
0=\sF_0 \subsetneq \sF_1 \subsetneq \dots \subsetneq \sF_k=\sF,
\]
where each quotient $\sQ_i=\sF_i/\sF_{i-1}$ is torsion free, $\alpha$-semistable and where the sequence of slopes $\mu_{\alpha}(\sQ_i)$ is strictly decreasing. In particular, the sheaf $\sF_1$ is exactly the maximal destabilizing subsheaf of $\sF$. Moreover, for each $1\leq i\leq k$, we have
\begin{equation}
	\label{e.min-max-equality}
	\mu_{\alpha}^{\min}(\sF_i) = \mu_{\alpha}(\sQ_i) = \mu_{\alpha}^{\max}(\sF/\sF_{i-1}).
\end{equation}

Now we suppose that $\mu_{\alpha}^{\max}(\sF)>0$ and set $s\coloneqq \max\{1\leq i\leq k\,|\,\mu_{\alpha}(\sQ_i)>0\}$. Then we define the \emph{positive part of $\sF$ with respect to $\alpha$} to be the sheaf $\sF_{\alpha}^+\coloneqq \sF_s$. 

\begin{lem}[\protect{\cite[Corollary 2.18]{AraujoDruel2019}}]
	\label{l.integrability}
	Let $\sF$ be a foliation on a normal $\bQ$-factorial projective variety $X$ with $\mu_{\alpha}^{\max}(\sF)\geq 0$ for some movable class $\alpha$. Then $\sF_i$ is a foliation on $X$ whenever $\mu_{\alpha}(\sQ_i)\geq 0$.
\end{lem}

\begin{thm}[\protect{\cite[Theorem 1.1]{CampanaPaun2019}}, compare it with Theorem \ref{t-CP-simplified}]
	\label{t.CP-Thm}
	Let $\sF$ be a foliation on a normal $\bQ$-factorial projective variety $X$. If $\mu_{\alpha}^{\min}(\sF)>0$ for some movable curve class $\alpha$, then $\sF$ is algebraically integrable and the closure of a general leaf is rationally connected.
\end{thm}

We also need the following useful criterion for uniruledness.

\begin{lem}[\protect{\cite[Theorem 2.7]{BoucksomDemaillyPuaunPeternell2013}}]
	\label{l.uniruledness-criterion}
	Let $X$ be a projective manifold. If there exists a foliation $\sF\subset T_X$ and a movable curve class $\alpha$ such that $\mu_{\alpha}^{\max}(\sF)>0$, then $X$ is uniruled.
\end{lem}

\begin{proof}
	Let $\sF'$ be the maximally destabilising subsheaf of $\sF$ with respect to $\alpha$. Then we have $c_1(\sF')\cdot \alpha>0$ and therefore $\det(\sF'^*)$ is not pseudoeffective. Applying \cite[Theorem 2.7]{BoucksomDemaillyPuaunPeternell2013} yields that $X$ is uniruled.
\end{proof}

\subsection{Minimal rational curves}

Let $X$ be a uniruled projective manifold. Then there exists a covering family $\cK$ of minimal rational curves; that is, an irreducible component of $\text{Ratcurves}^n(X)$ such that for a general point $x\in X$, the closed subset $\cK_x$ of $\cK$ parametrising curves through $x$ is non-empty and propre, see \cite{Kollar1996} for the details. Let $C$ be a general rational curve parametrised by $\cK$. Then $C$ is standard. In other words, there exists a non-negative integer $d$ such that
\[
f^*T_X \cong \sO_{\bP^1}(2)\oplus \sO_{\bP^1}(1)^{\oplus d} \oplus \sO_{\bP^1}^{\oplus (n-d-1)},
\]
where $f:\bP^1\rightarrow C$ is the normalisation. 

Given a covering family $\cK$ of minimal rational curves on a uniruled projective manifold $X$, let $\overline{\cK}$ be the closure of $\cK$ in $\Chow(X)$. Two points $x$, $y\in X$ are said to be \emph{$\cK$-equivalent} if they can be connected by a chain of $1$-cycles from $\overline{\cK}$. This defines an equivalence relation on $X$. By \cite{Campana1992} (see also \cite[IV, 4.16]{Kollar1996}), there exists a propre surjective equidimensional morphism $\pi^\circ:X^{\circ} \rightarrow T^{\circ}$ from a dense open subset of $X$ onto a normal variety whose fibres are $\cK$-equivalence classes. We call this map the \emph{$\cK$-rationally connected quotient} of $X$.

\begin{thm}[\protect{\cite{Araujo2006} and \cite[Proposition 2.7]{AraujoDruelKovacs2008}}]
	\label{t.Araujo-ProjBundle}
	Let $X$ be a uniruled projective manifold equipped with a covering family $\cK$ of minimal rational curves. Let $\pi^{\circ}:X^{\circ}\rightarrow T^{\circ}$ be the $\cK$-rationally connected quotient of $X$. Assume that there exists a subsheaf $\sF\subset T_X$ such that $f^*\sF$ is ample, where $f:\bP^1\rightarrow C$ is the normalisation of a general curve parametrised by $\cK$. Then, after shrinking $X^{\circ}$ and $T^{\circ}$ if necessary, $\pi^{\circ}$ becomes a $\bP^{d+1}$-bundle, $\sF|_{X^{\circ}}\subset T_{X^{\circ}/T^{\circ}}$ and every rational curve parametrised by $\cK$ meeting $X^{\circ}$ is a line contained in the fibres of $\pi^{\circ}$.
\end{thm}

\subsection{Proof of Theorem \ref{t.Seshadri-MaxValue}}

From now on let $X$ be a projective manifold such that there exists a foliation $\sF\subsetneq T_X$ with $-K_{\sF}$ ample. Then $X$ is uniruled by Lemma \ref{l.uniruledness-criterion}. Fix a covering family $\cK$ of minimal rational curves on $X$. Let $\alpha=[C]$ be the numerical class of a general minimal rational curve $C$ parametrised by $\cK$. Denote by $\sF_{\alpha}^+$ the positive part of $\sF$ with respect to $\alpha$. By Lemma \ref{l.integrability}, Theorem \ref{t.CP-Thm} and \eqref{e.min-max-equality}, the sheaf $\sF_{\alpha}^+$ is an algebraically integrable foliation on $X$ and therefore $\sF_{\alpha}^+\subset \sF^a$. In particular, we have $r^+\leq r^a$, where $r^+$ and $r^a$ are the ranks of $\sF_{\alpha}^+$ and $\sF^a$, respectively. Denote by $f:\bP^1\rightarrow C$ the normalisation of the standard rational curve $C$ and write 
\[
f^*T_X\cong \sO_{\bP^1}(2)\oplus \sO_{\bP^1}(1)^{\oplus d} \oplus\sO_{\bP^1}^{\oplus (n-d-1)}.
\]

\subsubsection{Step 1. Splitting type of $\sF^+_{\alpha}$}

We determine the possibilities of the splitting types of the positive part $\sF^+_{\alpha}$ along the general minimal rational curve $C$.

\begin{claim}
	\label{c.splitting-types}
If $\epsilon(-K_{\sF}) = r^a$, then one of the following statements holds.
	\begin{enumerate}
		\item \label{l.splitting-1}$r^a=r^+ + 1$ and $f^*\sF_{\alpha}^+\cong \sO_{\bP^1}(2)\oplus \sO_{\bP^1}(1)^{\oplus (r^+ - 1)}$.
		
		\item \label{l.splitting-2}$r^a=r^+$ and $f^*\sF_{\alpha}^+\cong \sO_{\bP^1}(2)\oplus \sO_{\bP^1}(1)^{\oplus (r^+ - 1)}$.
		
		\item \label{l.splitting-3}$r^a=r^+$ and $f^*\sF_{\alpha}^+ \cong \sO_{\bP^1}(2)\oplus \sO_{\bP^1}(1)^{\oplus (r^+ -2)} \oplus \sO_{\bP^1}$, $r^+\geq 2$.
		
		\item \label{l.splitting-4}$r^a=r^+$ and $f^*\sF_{\alpha}^+ \cong \sO_{\bP^1}(1)^{\oplus r^+}$.
	\end{enumerate}
\end{claim}

\begin{proof}[Proof of Claim \ref{c.splitting-types}]
	Let $\sQ$ be the quotient $\sF/\sF_{\alpha}^+$. Let $K_{\sQ}$ be a Weil divisor on $X$ such that $\sO_X(-K_{\sQ})\cong \det(\sQ)$. Then we have $K_{\sQ}\cdot \alpha\geq 0$ by the definition of $\sF_{\alpha}^+$ and \eqref{e.min-max-equality}. In particular, we get
	\[
	-K_{\sF_{\alpha}^+} \cdot \alpha \geq -K_{\sF} \cdot \alpha = -K_{\sF}\cdot C \geq  \epsilon(-K_{\sF}) = r^a.
	\]
	On the other hand, as $\sF_{\alpha}^+$ is saturated in $T_X$ and $C$ is general, by \cite[II, Proposotion 3.7]{Kollar1996}, we can assume that $f^*\sF_{\alpha}^+$ is a subbundle of $f^*T_X$. Write 
	\[
	f^*\sF_{\alpha}^+\cong \sO_{\bP^1}(a_1)\oplus \dots \oplus \sO_{\bP^1}(a_{r^+})
	\]
	with $a_1\geq \dots \geq a_{r^+}$. Since $C$ is standard, we have $a_1\leq 2$, $a_i\leq 1$ if $2\leq i\leq d+1$ and $a_i\leq 0$ if $i\geq d+2$, where $d=-K_X\cdot C - 2$. In particular, one obtains
	\[
	-K_{\sF_{\alpha}^+} \cdot C \leq r^+ + 1
	\] 
	with equality if and only if 
	\[
	f^*\sF_{\alpha}^+\cong \sO_{\bP^1}(2)\oplus \sO_{\bP^1}(1)^{\oplus (r^+-1)}.
	\]
	On the other hand, as $r^+\leq r^a$, we get
	\[
	r^a + 1\geq r^+ + 1 \geq -K_{\sF_{\alpha}^+}\cdot \alpha \geq  r^a.
	\]
	If $r^a=r^+ + 1$, then we have $-K_{\sF_{\alpha}^+}\cdot C=r^+ + 1$ and we are in Case \ref{l.splitting-1}. On the other hand , if $r^a=r^+$, then we have $r^+ + 1 \geq -K_{\sF_{\alpha}^+}\cdot C\geq r^+$ and the result follows from an easy computation.
\end{proof}

\subsubsection{Step 2. General minimal rational curves are not tangent to $\sF^+_{\alpha}$.}
 
	\begin{claim}
		\label{c.tangency-minimal-rational-curve}
		Let $C$ be a general minimal rational curve parametrised by $\cK$ with normalisation $f:\bP^1\rightarrow C$. Then $T_{\bP^1}$ is not contained in $f^*\sF_{\alpha}^+$.
	\end{claim}
	
	\begin{proof}[Proof of Claim \ref{c.tangency-minimal-rational-curve}]
		Let $\pi:U\rightarrow T$ be the family of leaves of $\sF^a$ and $\nu:U\rightarrow X$ the evaluation morphism. By \cite[II, Proposition 3.7]{Kollar1996}, we may assume that general minimal rational curves parametrised by $\cK$ are disjoint from $\nu(\Exc(\nu))$.
		
		Assume to the contrary that $T_{\bP^1}$ is contained in $f^*\sF^+_{\alpha}$. Then we have $T_{\bP^1}\subset f^*\sF^a$. In particular, there exists a general fibre $F$ of $\pi$ such that $C$ is contained in $\nu(F)$. Moreover, as $C$ is general, the preimage $\nu^{-1}(C)\subset F$ is disjoint from $\Exc(\nu)\cap F=\supp(\Delta_F)$ (see \cite[Lemma 2.12]{AraujoDruel2019}). Nevertheless, as $F\cong \bP^{r^a}$ has Picard number one and $\Delta_F$ is a hyperplane in $\bP^{r^a}$ by Corollary \ref{c.log-leaf-maxSeshadri}, we get a contradiction.
	\end{proof}

\subsubsection{Step 3. End of proof.}
	
	We end the proof by showing that $\sF$ is induced by a linear projection of $\bP^n$ to $\bP^{n-r^a}$. By Step 1 and Step 2, the splitting type $f^*\sF^+_{\alpha}$ must be of the form $\sO_{\bP^1}(1)^{\oplus r^+}$ and $\sF^a=\sF^+_{\alpha}$. Moreover, by Theorem \ref{t.Araujo-ProjBundle}, after shrinking $T^{\circ}$ if necessary, the $\cK$-rationally connected quotient $\pi^{\circ}:X^{\circ}\rightarrow T^{\circ}$ is a $\bP^{d+1}$-bundle such that $\sF^a|_{X^{\circ}}\subset T_{X^{\circ}/T^{\circ}}$. On the other hand, according to Corollary \ref{c.log-leaf-maxSeshadri}, there is a common point in the closure of a general leaf of $\sF^a$. Hence, we must have $\dim(T)=0$. As a consequence, we obtain $X\cong \bP^n$ and $\sF^a \subset T_{\bP^n}$ is a foliation with rank $r^a$ and $\det(\sF^a)=\sO_{\bP^n}(r^a)$. Moreover, as $\epsilon(\sO_{\bP^n}(1))=1$ and $\epsilon(-K_{\sF})=r^a$, we obtain 
	\[
	\det(\sF^a) \cong \det(\sF) \cong \sO_{\bP^n}(r^a).
	\]
	Hence, by \cite[Theorem 1.5 and Corollary 1.7]{AraujoDruel2019}, the foliation $\sF$ is the linear pull-back of a purely transcendental foliation on $\bP^{n-r^a}$ with zero canonical class. This finishes the proof.

\section{Examples and related questions}

\label{s.Examples-Seshadri}

In this last section, we exhibit some examples concerning the sharpness of our main results and also propose a few interesting related questions. We start with the following simple observation.

\begin{lem}
	\label{l.cone-projective-bundle}
	In Example \ref{e.generalised-index}, if we assume furthermore that $Z$ is smooth and $\Pic(Z)\cong \bZ\sO_Z(1)$, then we have
	\begin{center}
		$\Nef(X)=\langle \Lambda+b_1\pi^*A,\pi^*A\rangle$ and $\Pseff(X)=\langle E,\pi^*A\rangle$.
	\end{center}
\end{lem}

\begin{proof}
	The description of the nef cone of $X$ is obvious and we only need to deal with the pseudoeffective cone of $X$. Let $\alpha$ be the unique real number such that $\Lambda-\alpha\pi^*A$ generates an extremal ray of $\Pseff(X)$. As $E\sim \Lambda - m\pi^*A$ is effective, we must have $\alpha\geq m$. In particular, the restriction $(\Lambda-\alpha \pi^*A)|_E$ is not pseudoeffective. As a consequence, the divisor $\Lambda-\alpha \pi^*A$ is not modified nef and hence there exists a unique prime effective divisor $D'$ such that $[D']\in \bR_{> 0}(\Lambda-\alpha\pi^*A)$ by \cite[Lemma 2.5]{HoeringLiuShao2022}. In particular, as $D'|_E$ is not pseudoeffective, we obtain $D'=E$ and $\alpha=m$.
\end{proof}

\begin{example}
	\label{e.folia-arb-index}
	Let $(r,n)$ be a pair of positive integers such that $r<n$. Then for any integer $d\geq -r$, there exists a foliation $\sF$ on $\bP^n$ with algebraic rank $r^a=r$ and such that  $\sO_{\bP^n}(K_{\sF})\cong \sO_{\bP^{n}}(d)$.
	\begin{enumerate}
		\item If $r\leq n-2$, we may take $\pi:\bP^n\dashrightarrow \bP^{n-r}$ to be a linear projective and then take $\sG\subsetneq T_{\bP^{n-r}}$ to be a purely transcendental rank one foliation on $\bP^{n-r}$ such that 
		\[
		\sG\cong \sO_{\bP^{n-r}}(-d-r)\quad \text{and} \quad d+r\geq 0.
		\]
		Then the pull-back $\sF=\pi^{-1}\sG$ is a foliation with algebraic rank $r^a=r$ and such that $\det(\sF)\cong \sO_{\bP^n}(-d)$.
		
		\item If $r=n-1$, we may take $\pi:\bP^n\dashrightarrow \bP^1$ to be a rational map defined by two general coprime homogeneous polynomials $f$ with degree $d_f$ and $g$ with degree $d_g$ such that $d=d_f+d_g-n-1$. Then the foliation $\sF$ induced by $\pi$ has algebraic rank $r^a=n-1$ and $\det(\sF)\cong \sO_{\bP^n}(-d)$.
	\end{enumerate}
\end{example}

\subsection{Generalised index and Seshadri constant}

In this subsection, we will discuss the sharpness of Theorem \ref{t.Kobayashi-Ochiai-I} and Theorem \ref{t.bounding-Seshadri}. Firstly let us recall the following question asked by Araujo and Druel in \cite{AraujoDruel2019} on the generalised indices of foliations.

\begin{question}[\protect{\cite[Question 4.4]{AraujoDruel2019}}]
	\label{q.AD-question}
	Is there a foliation $\sF$ on a projective manifold $X$ with $\widehat{\iota}(\sF)\not\in \bN$ and $\widehat{\iota}(\sF)<r^a<\widehat{\iota}(\sF)+1$?
\end{question}

We will give a positive answer to this question by constructing foliations on the projective bundles $X$ given in Example \ref{e.generalised-index}. We start with the following general result.

\begin{lem}
	\label{l.generalised-index-general-formula}
	In Example \ref{e.generalised-index}, we assume furthermore that $Z$ is smooth and $\Pic(Z)\cong \bZ\sO_Z(1)$. Let $D\equiv \beta\Lambda + \gamma\pi^*A$ be a $\bR$-Cartier $\bR$-divisor on $X$. If $D$ is big but not ample, i.e., $-m\beta<\gamma \leq b_1\beta$, then we have
	\[
	\widehat{\iota}(D) =  \frac{\beta m + \gamma}{m+b_1+1}.
	\]
\end{lem}

\begin{proof}
	Write $D\equiv \widehat{\iota}(D) H + P$ for an ample Cartier divisor $H$ and a pseudoeffective $\bR$-Cartier $\bR$-divisor $P$. Let $a$ and $e$ be the unique non-negative real numbers such that $P\equiv eE+a\pi^*A$. If $e=0$, then $P\equiv a\pi^*A$ is nef. In particular, the divisor $D$ is ample, which is impossible by our assumption. Thus we must have $e>0$. On the other hand, if $a>0$, then $P$ is big and consequently there exists a positive real number $t$ such that $P-tH$ is pseudoeffective. In particular, we get
    \[
    D \equiv \widehat{\iota}(D) H + P \equiv (\widehat{\iota}(D)+t)H + (P-tH),
    \]
    which is absurd. Hence, we have $a=0$ and $P\equiv eE$ with $e>0$.
    
    Since $H$ is an ample Cartier divisor, by the description of $\Nef(X)$, there exists two positive integers $d$ and $c>b_1 d$ such that $H\equiv d\Lambda + c\pi^*A$. By definition, we have
    \begin{align*}
    	D - P \equiv D - eE 
    	    & \equiv (\beta-e)\Lambda + (em + \gamma)\pi^*A \\
    	    & \equiv \widehat{\iota}(D)(d\Lambda + c\pi^*A).
    \end{align*}
    Since $\Lambda$ and $\pi^*A$ are linearly independent, this implies
    \begin{equation}
    	\label{e.iota}
    	\widehat{\iota}(D) = \frac{\beta-e}{d}
    \end{equation}
    and
    \begin{equation}
    	\label{e.iota-lower-bound}
    	em + \gamma = \widehat{\iota}(D)c \geq \widehat{\iota}(D)(b_1 d+1).
    \end{equation}
    Then combining \eqref{e.iota} and \eqref{e.iota-lower-bound} yields
    \begin{equation}
    	(md + b_1 d +1)e  \geq \beta(b_1 d+1) - d\gamma.
    \end{equation}
    In particular, we obtain
    \begin{equation}
    	\label{e.upper-bound-P}
    	e\geq \frac{\beta(d b_1 + 1) -d\gamma}{md + b_1 d + 1}.
    \end{equation}
    Then combining \eqref{e.upper-bound-P} with \eqref{e.iota} gives
    \begin{equation}
    	\widehat{\iota}(D) \leq \frac{\beta(md+b_1 d+1)-\beta(b_1 d+1)+d\gamma}{d(md + b_1 d+1)} \leq  \frac{\beta m +\gamma}{m+b_1+1}
    \end{equation}
    and the equality holds if and only if $d=1$ and $c=b_1+1$.
\end{proof}

\begin{prop}
	\label{p.examples-generalised-index}
	Let $(r,n)$ be a pair of positive integers such that $r<n$ and $n\geq 3$. Then for any positive rational number $c\leq r$, there exists an $n$-dimensional projective manifold $X$ and a foliation $\sF\subsetneq T_X$ with algebraic rank $r^a=r$ such that $-K_{\sF}$ is big and
	\[
	\widehat{\iota}(\sF)=\epsilon(\widehat{\iota}(\sF)H)=c,
	\]
	where $H$ is an ample Cartier divisor $H$ such that $P\coloneqq -K_{\sF}-\widehat{\iota}(\sF)H$ is pseudoeffective.
\end{prop}

\begin{proof}
	 By Example \ref{e.folia-arb-index}, we may assume that $c\not\in \bZ$. In particular, there exist two positive integers $p$ and $q$ such that $c=p/q<r$.
	 \\[10pt]
	 \textbf{Case 1.} \textit{$r>c>1$, i.e. $p>q$.} Let $l$ be a large enough positive integer such that
	 \begin{equation}
	 	\label{e.choice-l}
	 	l(p-q)+(p-qr)+1>0\quad \text{and}\quad \frac{p}{q}\leq \left(1-\frac{1}{ql}\right)r + \frac{1}{ql}.
	 \end{equation}
	 Note that this is possible because $q<p<qr$. Let $b_1=lq-1$ and let $b_i$ $(2\leq i\leq r)$ be some non-negative integers such that
	 \begin{equation}
	 	b-b_1=\sum_{i=2}^r b_i= l(p-q) + (p-qr) + 1.
	 \end{equation}
     The existence of $b_i$'s follows from the fact that by \eqref{e.choice-l} we have $b-b_1>0$ and 
     \[
     b-b_1 = l(p-q) + (p-qr) + 1 \leq l(p-q) \leq (lq-1)(r-1) = b_1(r-1).
     \]
     Let $(Z,\sO_Z(1))$ be the pair $(\bP^{n-r},\sO_{\bP^{n-r}}(1))$ and let $\pi:X\rightarrow Z$ be the projective bundle constructed in Example \ref{e.generalised-index} with $m=q$, $r'=r$ and $b_i$'s chosen as above. Let $\sF$ be the algebraically integrable foliation induced by $\pi$. Then we have
     \[
     -K_{\sF} = -K_{X/Z} \sim (r+1)\Lambda - (q-b)\pi^*A.
     \]
     In particular, $-K_{\sF}$ is big but not ample,  and by Lemma \ref{l.generalised-index-general-formula} we have
     \begin{equation}
     	\widehat{\iota}(\sF) = \widehat{\iota}(-K_{\sF}) = \frac{(r+1)q + b - q}{q+lq} = \frac{p}{q}.
     \end{equation}
	 \textbf{Case 2.} \textit{$0<c<1$, i.e. $p<q$.} Let $(Z,\sO_Z(1))$ be the pair $(\bP^{n-1},\sO_{\bP^{n-1}}(1))$ and let $X$ be the $n$-dimensional projective manifold constructed in Example \ref{e.generalised-index} with $m$ and $b=b_1$ to be determined. Let $\sG$ be a foliation on $Z$ with algebraic rank $r-1$ and $\sO_Z(K_{\sG})\cong \sO_Z(d)$ with $d\geq -r+1$ to be determined. Denote by $\sF$ the pull-back $\pi^{-1}\sG$. Then by \cite[2.6]{AraujoDruel2019}, we have
	 \begin{equation}
	 	- K_{\sF} = - K_{X/Z} - \pi^*K_{\sG} = 2\Lambda + (b-m-d)\pi^*A.
	 \end{equation}
     Note that $-K_{\sF}$ is big but not ample if and only if $-2m<b-m-d\leq 2b$ by Lemma \ref{l.cone-projective-bundle}. Moreover, if so, then Lemma \ref{l.generalised-index-general-formula} implies
     \begin{equation}
     	\label{e.index-c-less-1}
     	\widehat{\iota}(\sF) = \frac{2m + b-m-d}{m+b+1} =\frac{m+b-d}{m+b+1}. 
     \end{equation}
     Set $m=q$, $b=q-1$ and let $d=2(q-p)-1\geq 1\geq 1-r$. Then the existence of  $(\bP^{n-1},\sG)$ follows from Example \ref{e.folia-arb-index}. Moreover, it is easy to see that $-2m<b-m-d\leq 2b$ and therefore by \eqref{e.index-c-less-1} a straightforward computation  shows that
     \begin{equation}
     	\widehat{\iota}(\sF) = \frac{q+q-1-(2(q-p)-1)}{q+q-1+1} =\frac{p}{q}.
     \end{equation}
     \\[10pt]
     Finally note that the line bundle $\sO_{\bP(\sE)}(1)\otimes \pi^*\sO_{Z}(b_1+1)$ is very ample in both cases and consequently we have $\epsilon(H)\geq 1$, where $H=\Lambda + (b_1+1)\pi^*A$. On the other hand, since the restriction of $H$ to fibres of $\pi$ is a hyperplane section, we obtain $\epsilon(H)=1$. Moreover, by the proof of Lemma \ref{l.generalised-index-general-formula}, in both cases we have $-K_{\sF} \equiv \widehat{\iota}(\sF)H+eE$ for some positive rational number $e>0$. This finishes the proof.
\end{proof}

\begin{example}
	\label{e.curve-foliations}
	Let $a\geq 2$ be a positive integer and let $\pi:X\rightarrow \bP^1$ be the Hirzebruch surface $\bP(\sO_{\bP^1}(a-1)\oplus \sO_{\bP^1})$. Denote by $\sF$ the foliation induced by $\pi$. Then we have 
	\[
	-K_{\sF} \sim 2\Lambda - (a-1)\pi^*A,
	\]
	where $A$ is a point on $\bP^1$. Then $-K_{\sF}$ is big but not ample by Lemma \ref{l.cone-projective-bundle} and by Lemma \ref{l.generalised-index-general-formula} we have
	\[
	\widehat{\iota}(-K_{\sF}) = \epsilon(\widehat{\iota}(\sF) H) = \frac{2(a-1)- (a-1)}{a} = 1 -\frac{1}{a},
	\]
	where $H=\Lambda + \pi^*A$ is an ample Cartier divisor such that $-K_{\sF}-\widehat{\iota}(\sF) H$ is pseudoeffective. As a consequence, combining this example with Proposition \ref{p.examples-generalised-index} gives a positive answer to Question \ref{q.AD-question}.
\end{example}

\begin{question}
	Given any positive rational number $c<1$, does there exist a smooth projective surface $X$ and a foliation $\sF$ on $X$ such that $-K_{\sF}$ is big and $\widehat{\iota}(\sF)=c$? By Example \ref{e.curve-foliations}, the answer is positive if $c$ is contained in the \emph{standard multiplicities} 
	\[
	\Phi\coloneqq \left\{1-\frac{1}{a}\,|\, a\in\bZ_{>1}\right\}.
	\]
\end{question}

\subsection{Fano index and Seshadri constant}

In this subsection, we will discuss the sharpness of Theorem \ref{t.Kobayashi-Ochiai-II}, Theorem \ref{c.upper-bound-Seshadri-weak-Fano} and Theorem \ref{t.Seshadri-MaxValue}. Let $\sF\subsetneq T_X$ be a Fano foliation on a normal projective variety $X$. By definition, it is clear that we have $\iota(\sF)\leq \widehat{\iota}(\sF)$. The following example shows that this inequality may be strict in general case.

\begin{example}
	Let $r\geq 2$ be a fixed positive integer. In Example \ref{e.generalised-index}, let $r'=r$, $m=1$ and $b_i=r-2$ for any $1\leq i\leq r'$. Moreover, by Example \ref{e.folia-arb-index}, we may choose a foliation $\sG$ on $\bP^{n-r}$ such that $\det(\sG)\cong \sO_{\bP^n}(r)$ and $2r<n-1$. Denote by $\sF$ the pull-back $\pi^{-1}\sG$. Then we have
	\begin{align*}
		K_{\sF} 
		    & \sim -(r+1)\Lambda - \pi^*(r(r-2)-1)A + \pi^*K_{\sG} \\
		    &  \sim -(r+1)\Lambda - ((r+1)(r-2)+1)\pi^*A.
	\end{align*}
	In particular, by Lemma \ref{l.cone-projective-bundle}, $-K_{\sF}$ is ample. On the other hand, note that we have
	\[
	-K_{\sF} - E \sim r\Lambda + ((r+1)(r-2) + 2)) \pi^*A = r (\Lambda + (r-1)\pi^*A)
	\]
	and therefore $\widehat{\iota}(\sF)\geq r$. On the other hand, let $H=d\Lambda + c\pi^*A$ be an ample Cartier divisor such that $d$ and $c$ are two positive integers satisfying 
	\begin{center}
		$c\geq b_1d+1=d(r-2)+1$ and $-K_{\sF} \sim_{\bQ} \iota(\sF) H$.
	\end{center}
    Then we obtain
    \begin{center}
    	$\iota(\sF)d=r+1$ and $\iota(\sF)c=(r+1)(r-2)+1$.
    \end{center}
    This yields
    \[
    (r+1)(r-2)+1 = \iota(\sF) c \geq \iota(\sF)(d(r-2)+1) \geq (r+1)(r-2) + \iota(\sF).
    \]
    Thus we obtain $\iota(\sF)\leq 1$ and hence $\iota(\sF)=1< r\leq \widehat{\iota}(\sF)$.
\end{example}

\begin{lem}
	\label{l.folia-generalised-cone}
	In Example \ref{e.generalised-cone}, we assume furthermore that $Z$ is smooth and $\sO_Z(1)$ is globally generated. Let $\sH$ be a foliation on $Z$ with $\sO_Z(K_{\sH})\cong \sO_Z(d)$ and let $\sF$ be the foliation  induced by $\sH$ on the normal generalised cone $Y$ over $(Z,\sO_Z(m))$ with vertex $\bP^{r'}$. If $d<mr'$, then $\sF$ is a Fano foliation on $Y$ such that
	\[
	\epsilon(-K_{\sF}) = \iota(\sF) =\widehat{\iota}(\sF) = r' - \frac{d}{m}.
	\]
\end{lem}

\begin{proof}
	We follow the notations in Example \ref{e.generalised-index} and Example \ref{e.generalised-cone}. Let $H$ be the ample Cartier divisor on $Y$ such that $\mu^*H=\Lambda$. As $Y$ is covered by rational curves of $H$-degree one and $|\Lambda|$ is base-point free, we have $\epsilon(H)=\epsilon(\Lambda)=1$. Let $\sG$ be the pull-back foliation $\pi^{-1}\sH$ on $X$. By \cite[2.6]{AraujoDruel2019}, we have
	\begin{equation}
		K_{\sG} = K_{X/Z} + \pi^*K_{\sH} \sim -(r'+1)\Lambda + (m+d)\pi^*A.
	\end{equation}
    As $m\mu_*\pi^*A \sim H$ and $\mu_*\Lambda = H$, one gets
    \begin{equation}
    	-K_{\sF} = -\mu_*K_{\sG} \sim_{\bQ} \left(r'+1-\frac{m+d}{m}\right) H = \left(r'-\frac{d}{m}\right)H.
    \end{equation}
    Hence, the $\bQ$-Cartier divisor $-K_{\sF}$ is ample if and only if $d<mr'$. Moreover, if so, then we have $\iota(\sF)=\epsilon(-K_{\sF})= \widehat{\iota}(\sF)=r'-d/m$ since $Y$ is covered by rational curves with $H$-degree one. 
\end{proof}

\begin{prop}
	\label{p.Sharpness-Fano-Folia}
	Let $(r,n)$ be a pair of positive integers such that $r<n$. For any rational number $c\leq \min\{r,n-2\}$, there exists an $n$-dimensional $\bQ$-factorial normal projective variety $X$ with klt singularities and a Fano foliation $\sF\subsetneq T_X$ with algebraic rank $r^a=r$ such that 
		\[
		\iota(\sF)= \widehat{\iota}(\sF) = \epsilon(-K_{\sF})=c.
		\]	
\end{prop}

\begin{proof}
	By Example \ref{e.folia-arb-index}, we may assume that $c\not\in \bZ$. Write $\lceil c\rceil - c =p/q<1$ for some positive integers $p$ and $q$. Denote $\lceil c\rceil$ by $r'$. Then we have $n-r'\geq 2$. Let $Y$ be the normal generalised cone over the base $(\bP^{n-r'},\sO_{\bP^{n-r'}}(q))$ with vertex $\bP^{r'-1}$. Let $\sH$ be a foliation on $\bP^{n-r'}$ with algebraic rank $r-r'$ and such that $\sO_{\bP^{n-r'}}(K_{\sH})\cong \sO_{\bP^{n-r'}}(p)$ (cf. Example \ref{e.folia-arb-index}). Denote by $\sF$ the foliation on $Y$ induced by $\sH$. Then $\sF$ has algebraic rank $r$ and by Lemma \ref{l.folia-generalised-cone} we have
	\[
	\epsilon(-K_{\sF}) = \iota(\sF) = \widehat{\iota}(\sF) = r' - \frac{p}{q} = \lceil c\rceil - \frac{p}{q} = c.
	\]
	This finishes the proof.
\end{proof}

\begin{rem}
	In general the foliations constructed in Proposition \ref{p.Sharpness-Fano-Folia} are very far from being log canonical in the sense of McQuillan (\cite[3.9]{AraujoDruel2013}).
\end{rem}

In the following we consider codimension one algebraically integrable foliations.

\begin{example}
	\label{e.codimension-foliation}
	Let $X$ be the weighted projective space $\bP(1,a_1,\dots,a_n)$ where $a_1\leq \dots \leq a_n$  are positive positive integers satisfying $\text{gcd}(a_1,\dots,a_n)=1$. Then $X$ is $\bQ$-factorial Fano variety of Picard number one and with klt singularities. Let $\sF$ be the foliation induced by the rational map $X\dashrightarrow \bP^1$ by sending $[x_0:x_1:\dots:x_n]$ to $[x_0^{a_1}:x_1]$. Then we have
	\[
	-K_{\sF} \sim_{\bQ} \left(\sum_{i=2}^n a_n\right) H,
	\]
	where $H$ is the prime divisor on $X$ defined as $\{x_0=0\}$. Moreover, note that the Cartier index of $H$ is $\text{lcm}(a_1,\dots,a_n)$ and by \cite[Example 22]{LiuZhuang2018} we have $\epsilon(H)=1/a_n$.
	\begin{enumerate}
		\item Assume that $n\geq 3$. Set $a_1=a_2=1$ and $a_3=\dots=a_n=m$. Then $\sF$ is a Fano foliation satisfying
		\[
		\widehat{\iota}(\sF)=\iota(\sF)=\varepsilon(-K_{\sF})=n-2+\frac{1}{m}.
		\]
		We remark that here $X$ is nothing but the normal generalised cone over the base $(\bP^2,\sO_{\bP^2}(m))$ with vertex $\bP^{n-3}$ and $\sF$ is the obtained as the pull-back of the foliation by lines on $\bP^2$ induced by a linear projection $\bP^2\dashrightarrow \bP^1$.
		
		\item Assume that $n\geq 3$. Set $a_1=\dots=a_{n-1}=m'$ and $a_n=m$. Then $\sF$ is a Fano foliation satisfying
		\[
		\widehat{\iota}(\sF)=\iota(\sF)=\frac{(n-2)m'+m}{m'm} \quad \text{and}\quad \epsilon(-K_{\sF})=1+\frac{(n-2)m'}{m}.
		\]
		In particular, for any rational number $n-2<c<n-1$, we may choose suitable coprime positive integers $m'$ and $m$ such that $m'<m$ and $\epsilon(-K_{\sF})=c$.
		
		\item Assume that $n=2$. Then $\sF$ is a Fano foliation on the surface $X$ satisfying
		\[
	    \widehat{\iota}(\sF) = \iota(\sF) = \frac{1}{a_1} \quad \text{and}\quad \epsilon(-K_{\sF})=1.
		\]
		
		\item Assume that $n=2$ and let $\sG$ be the foliation induced by the rational map $X\dashrightarrow \bP^1$ by sending $[x_0:x_1:x_2]$ to $[x_0^{a_2}:x_2]$. Then $\sG$ is a Fano foliation such that $-K_{\sG}=a_1 H$. In particular, we have
		\[
		\widehat{\iota}(\sG)=\iota(\sG) = \frac{1}{a_2} \quad\text{and} \quad \epsilon(-K_{\sG})=\frac{a_1}{a_2}.
		\]
		As a consequence, for any rational number $0<c<1$, we may choose suitable coprime integers $a_1$ and $a_2$ such that $a_1<a_2$ and $\epsilon(-K_{\sG})=c$.
	\end{enumerate}
\end{example}

\begin{question}
	\label{q.Fano-index-Seshadri}
	According to Proposition \ref{p.Sharpness-Fano-Folia} and Example \ref{e.codimension-foliation}, it is natural to ask the following question. Given a positive integer $n\geq 2$ and a rational number $n-2<c<n-1$, does there exist Fano foliations $\sF$ on an $n$-dimensional projective variety $X$ such that $\widehat{\iota}(\sF)=\iota(\sF)=c$? Note that by Example \ref{e.codimension-foliation} we have a positive answer if $c$ is contained in the following set 
		\[
		\{n-2+\frac{1}{a}\,|\,a\in\bZ_{>0}\}.
		\]
\end{question}

Now we are in the position to finish the proof of Proposition \ref{p.existence-examples}.

\begin{proof}[Proof of Proposition \ref{p.existence-examples}]
	It follows from Proposition \ref{p.examples-generalised-index}, Proposition \ref{p.Sharpness-Fano-Folia} and Example \ref{e.codimension-foliation}.
\end{proof}

\subsection{Rationally connectedness and Seshadri constant}
\label{s.RC-Seshadri}

In this last subsection we discuss the sharpness ofTheorem \ref{c.RC-Folia-large-Seshadri} and Theorem \ref{t.rationally-connectedness}. Let $X\subset \bP^{n+1}$ be the cone over a plane curve of degree three. Then $X$ is a Fano variety with lc singularities such that $\epsilon(-K_X)=n-1$, but $X$ is not rationally connected. In particular, this shows that Theorem \ref{t.Zhuang} (1) is sharp. On the other hand, it is shown by Araujo and Druel in \cite[Theorem 1.8]{AraujoDruel2019} that the general leaf of the algebraic part of a foliation $\sF\subsetneq T_X$ on a projective manifold $X$ is rationally connected if $\widehat{\iota}(\sF)\geq r^a-1$. 

\begin{example}
	\label{e.RC-sharpness}
	Let $(r,n)$ be a pair of positive integers such that $1<r<n$. Let $Y$ be the normal generalised cone over $(\bP^{n-r+1},\sO_{\bP^{n-r+1}}(m))$ for some positive integer $m$. Let $\sH$ be an algebraically integrable foliation by curves on $\bP^{n-r+1}$ such that $\sH\cong \sO_{\bP^{n-r+1}}(-d)$ for some $d>0$ and the closure of the general leaf of $\sH$ is a smooth curve of genus at least two. Let $\sF$ be the foliation on $Y$ induced by $\sH$. If $d<m(r-1)$, then $\sF$ is an algebraically integrable Fano foliation with rank $r$ and by Lemma \ref{l.folia-generalised-cone} we have
	\[
	\epsilon(-K_{\sF}) = \iota(\sF) =\widehat{\iota}(\sF) = r-1 - \frac{d}{m}.
	\]
	In particular, if $m$ tends to infinity, then $\epsilon(-K_{\sF})$ tends to $r-1$ from left. Moreover, the closure of the general leaf of $\sF$ is actually a normal generalised cone over a smooth curve with genus at least two and hence it is not rationally connected.
\end{example}

\begin{example}
	\label{e.RC-sharpness-lc}
	Let $(r,n)$ be a pair of positive integers such that $1<r<n$. Let $W$ be an $(n-r)$-dimensional projective manifold with $K_W\equiv 0$ and let $C$ be an elliptic curve. Set $Z=C\times W$ and denote by $\sH\cong\sO_Z$ the foliation by curves on $Z$ induced by the projection $Z=C\times W\rightarrow W$. Let $Y$ be the normal generalised cone over $(Z,\sO_Z(m))$, where $\sO_Z(1)$ is a very ample line bundle over $Z$. Then $Y$ has only lc singularities. Let $\sF$ be the foliation on $Y$ induced by $\sH$. Then $\sF$ is an algebraically integrable foliation whose general leaves are not rationally connected. Moreover, by Lemma \ref{l.folia-generalised-cone} we have
	\[
	\epsilon(-K_{\sF}) = \iota(\sF) =\widehat{\iota}(\sF) = r-1.
	\]
\end{example}

The following question is communicated to me by the anonymous referee.

\begin{question}
	Let $\sF\subsetneq T_X$ be a foliation with positive enough tangent sheaf on a projective variety $X$ with mild singularities. Does $\sF$ have mild singularities? We refer the reader to \cite[\S\,3]{AraujoDruel2013} for the different definitions of singularities of foliations.
\end{question}

In Theorem \ref{t.Seshadri-MaxValue}, we have classified Fano foliations on smooth projective varieties with maximal Seshadri constant and the smoothness plays a key role in the proof. On the other hand, in the viewpoint of Kobayashi-Ochiai's theorem for foliation (cf. Theorem \ref{t.Kobayashi-Ochiai-II}), it is natural to ask the same question for singular varieties.

\begin{question}
	Classify those Fano foliations $\sF$ on $\bQ$-factorial normal projective varieties (with klt singularities) such that $\epsilon(-K_{\sF})=r^a$.
\end{question}
	
\bibliographystyle{alpha}
\bibliography{FanoFoliSeshadri}

\begin{thebibliography}{BDPP13}

\bibitem[AD13]{AraujoDruel2013}
Carolina Araujo and St{\'e}phane Druel.
\newblock On {F}ano foliations.
\newblock {\em Adv. Math.}, 238:70--118, 2013.

\bibitem[AD14]{AraujoDruel2014}
Carolina Araujo and St\'{e}phane Druel.
\newblock On codimension 1 del {P}ezzo foliations on varieties with mild
  singularities.
\newblock {\em Math. Ann.}, 360(3-4):769--798, 2014.

\bibitem[AD19]{AraujoDruel2019}
Carolina Araujo and St\'{e}phane Druel.
\newblock Characterization of generic projective space bundles and algebraicity
  of foliations.
\newblock {\em Comment. Math. Helv.}, 94(4):833--853, 2019.

\bibitem[ADK08]{AraujoDruelKovacs2008}
Carolina Araujo, St{\'e}phane Druel, and S{\'a}ndor~J. Kov{\'a}cs.
\newblock Cohomological characterizations of projective spaces and
  hyperquadrics.
\newblock {\em Invent. Math.}, 174(2):233--253, 2008.

\bibitem[Ara06]{Araujo2006}
Carolina Araujo.
\newblock Rational curves of minimal degree and characterizations of projective
  spaces.
\newblock {\em Math. Ann.}, 335(4):937--951, 2006.

\bibitem[BDPP13]{BoucksomDemaillyPuaunPeternell2013}
S{\'e}bastien Boucksom, Jean-Pierre Demailly, Mihai P{\u a}un, and Thomas
  Peternell.
\newblock The pseudo-effective cone of a compact {K}\"ahler manifold and
  varieties of negative {K}odaira dimension.
\newblock {\em J. Algebraic Geom.}, 22(2):201--248, 2013.

\bibitem[Bir12]{Birkar2012}
Caucher Birkar.
\newblock Existence of log canonical flips and a special {LMMP}.
\newblock {\em Publ. Math. Inst. Hautes \'Etudes Sci.}, 115:325--368, 2012.

\bibitem[BLR95]{BoschLuetkebohmertRaynaud1995}
Siegfried Bosch, Werner L{\"u}tkebohmert, and Michel Raynaud.
\newblock Formal and rigid geometry. {IV}. {T}he reduced fibre theorem.
\newblock {\em Invent. Math.}, 119(2):361--398, 1995.

\bibitem[Bou04]{Boucksom2004}
S\'{e}bastien Boucksom.
\newblock Divisorial {Z}ariski decompositions on compact complex manifolds.
\newblock {\em Ann. Sci. \'{E}cole Norm. Sup. (4)}, 37(1):45--76, 2004.

\bibitem[BS09]{BauerSzemberg2009}
Thomas Bauer and Tomasz Szemberg.
\newblock Seshadri constants and the generation of jets.
\newblock {\em J. Pure Appl. Algebra}, 213(11):2134--2140, 2009.

\bibitem[Cam92]{Campana1992}
Fr{\'e}d{\'e}ric Campana.
\newblock Connexit\'e rationnelle des vari\'et\'es de {F}ano.
\newblock {\em Ann. Sci. \'Ecole Norm. Sup. (4)}, 25(5):539--545, 1992.

\bibitem[CP19]{CampanaPaun2019}
Fr\'{e}d\'{e}ric Campana and Mihai P\u{a}un.
\newblock Foliations with positive slopes and birational stability of orbifold
  cotangent bundles.
\newblock {\em Publ. Math. Inst. Hautes \'{E}tudes Sci.}, 129:1--49, 2019.

\bibitem[Deb01]{Debarre2001}
Olivier Debarre.
\newblock {\em Higher-dimensional algebraic geometry}.
\newblock Universitext. Springer-Verlag, New York, 2001.

\bibitem[Dem92]{Demailly1992}
Jean-Pierre Demailly.
\newblock Singular {H}ermitian metrics on positive line bundles.
\newblock In {\em Complex algebraic varieties ({B}ayreuth, 1990)}, volume 1507
  of {\em Lecture Notes in Math.}, pages 87--104. Springer, Berlin, 1992.

\bibitem[Dru17]{Druel2017b}
St\'{e}phane Druel.
\newblock On foliations with nef anti-canonical bundle.
\newblock {\em Trans. Amer. Math. Soc.}, 369(11):7765--7787, 2017.

\bibitem[Fuj89]{Fujita1989}
Takao Fujita.
\newblock Remarks on quasi-polarized varieties.
\newblock {\em Nagoya Math. J.}, 115:105--123, 1989.

\bibitem[GKP16]{GrebKebekusPeternell2016}
Daniel Greb, Stefan Kebekus, and Thomas Peternell.
\newblock Movable curves and semistable sheaves.
\newblock {\em Int. Math. Res. Not. IMRN}, (2):536--570, 2016.

\bibitem[Gro66]{Grothendieck1966}
Alexander Grothendieck.
\newblock \'{E}l\'ements de g\'eom\'etrie alg\'ebrique. {IV}. \'{E}tude locale
  des sch\'emas et des morphismes de sch\'emas. {III}.
\newblock {\em Inst. Hautes \'Etudes Sci. Publ. Math.}, (28):255, 1966.

\bibitem[HLS22]{HoeringLiuShao2022}
Andreas H{\"o}ring, Jie Liu, and Feng Shao.
\newblock Examples of {F}ano manifolds with non-pseudoeffective tangent bundle.
\newblock {\em J. Lond. Math. Soc. (2)}, 106(1):27--59, 2022.

\bibitem[H{\"o}r14]{Hoering2014}
Andreas H{\"o}ring.
\newblock Twisted cotangent sheaves and a {K}obayashi-{O}chiai theorem for
  foliations.
\newblock {\em Ann. Inst. Fourier (Grenoble)}, 64(6):2465--2480, 2014.

\bibitem[KM98]{KollarMori1998}
J{\'a}nos Koll{\'a}r and Shigefumi Mori.
\newblock {\em Birational geometry of algebraic varieties}, volume 134 of {\em
  Cambridge Tracts in Mathematics}.
\newblock Cambridge University Press, Cambridge, 1998.

\bibitem[KO73]{KobayashiOchiai1973}
Shoshichi Kobayashi and Takushiro Ochiai.
\newblock Characterizations of complex projective spaces and hyperquadrics.
\newblock {\em J. Math. Kyoto Univ.}, 13:31--47, 1973.

\bibitem[Kol96]{Kollar1996}
J{\'a}nos Koll{\'a}r.
\newblock {\em Rational curves on algebraic varieties}, volume~32 of {\em
  Ergeb. Math. Grenzgeb., 3. Folge}.
\newblock Berlin: Springer-Verlag, 1996.

\bibitem[Kol97]{Kollar1997}
J{\'a}nos Koll{\'a}r.
\newblock Singularities of pairs.
\newblock In {\em Algebraic geometry---{S}anta {C}ruz 1995}, volume~62 of {\em
  Proc. Sympos. Pure Math.}, pages 221--287. Amer. Math. Soc., Providence, RI,
  1997.

\bibitem[Laz04]{Lazarsfeld2004}
Robert Lazarsfeld.
\newblock {\em Positivity in algebraic geometry. {I}. {Classical} setting: line
  bundles and linear series}, volume~48 of {\em Ergeb. Math. Grenzgeb., 3.
  Folge}.
\newblock Berlin: Springer, 2004.

\bibitem[Liu19]{Liu2019}
Jie Liu.
\newblock Characterization of projective spaces and $\mathbb{P}^{r}$ -bundles
  as ample divisors.
\newblock {\em Nagoya Math. J.}, 233:155--169, 2019.

\bibitem[LPT18]{LorayPereiraTouzet2018}
Frank Loray, Jorge~Vit\'{o}rio Pereira, and Fr\'{e}d\'{e}ric Touzet.
\newblock Singular foliations with trivial canonical class.
\newblock {\em Invent. Math.}, 213(3):1327--1380, 2018.

\bibitem[LZ18]{LiuZhuang2018}
Yuchen Liu and Ziquan Zhuang.
\newblock Characterization of projective spaces by {S}eshadri constants.
\newblock {\em Math. Z.}, 289(1-2):25--38, 2018.

\bibitem[Mae90]{Maeda1990}
Hidetoshi Maeda.
\newblock Ramification divisors for branched coverings of $\mathbb{P}^n$.
\newblock {\em Math. Ann.}, 288(2):195--199, 1990.

\bibitem[Mat89]{Matsumura1989}
Hideyuki Matsumura.
\newblock {\em Commutative ring theory}, volume~8 of {\em Cambridge Studies in
  Advanced Mathematics}.
\newblock Cambridge University Press, Cambridge, second edition, 1989.
\newblock Translated from the Japanese by M. Reid.

\bibitem[Nak04]{Nakayama2004}
Noboru Nakayama.
\newblock {\em Zariski-decomposition and abundance}, volume~14 of {\em MSJ
  Memoirs}.
\newblock Mathematical Society of Japan, Tokyo, 2004.

\bibitem[Pet94]{Peternell1994}
Thomas Peternell.
\newblock Minimal varieties with trivial canonical classes. {I}.
\newblock {\em Math. Z.}, 217(3):377--405, 1994.

\bibitem[Zhu18]{Zhuang2018a}
Ziquan Zhuang.
\newblock Fano varieties with large {S}eshadri constants.
\newblock {\em Adv. Math.}, 340:883--913, 2018.

\end{thebibliography}
\end{document}